\documentclass[12pt]{amsart}
\usepackage{amsmath}
\usepackage{amscd}
\usepackage{amssymb}
\usepackage{amsfonts}
\usepackage{amsthm}
\usepackage{caption}
\usepackage{comment}
\usepackage{float}
\usepackage{fullpage}
\usepackage[numbers]{natbib}  
\usepackage[hidelinks, draft=false]{hyperref}
\usepackage{xcolor}
\usepackage[toc]{appendix}

\theoremstyle{plain}
\newtheorem{prop}{Proposition}
\newtheorem{thrm}[prop]{Theorem}
\newtheorem{lem}[prop]{Lemma}
\newtheorem{conj}[prop]{Conjecture}

\theoremstyle{definition}
\newtheorem{defn}[prop]{Definition}
\newtheorem{rem}[prop]{Remark}
\newtheorem{ex}[prop]{Example}

\def\Q{{\mathbb Q}}
\def\cL{{\mathcal{L}}}
\def\cM{{\mathcal{M}}}
\def\Z{{\mathbb Z}}
\def\fM{{\mathfrak M}}
\def\C{{\mathbb C}}
\def\cO{{\mathcal O}}
\def\G{{\mathrm G}}
\def\GL{{\mathrm{GL}}}
\def\Maass{{\mathrm{Maass}}}
\def\MW{{\mathrm{Miyawaki}}}
\def\O{{\mathrm O}}

\def\SL{{\mathrm{SL}}}
\def\SO{{\mathrm{SO}}}
\def\SU{{\mathrm{SU}}}
\def\Sp{{\mathrm{Sp}}}
\def\spin{{\mathrm{spin}}}
\def\Spin{{\mathrm{Spin}}}
\def\U{{\mathrm{U}}}
\def\R{{\mathbb R}}
\def\cN{{\mathcal N}}

\def\Eis{{\mathrm{Eis}}}
\def\KE{{\mathrm{K}}}

\numberwithin{equation}{section}

\begin{document}
\title{Hermitian modular forms and algebraic modular forms on $\mathrm{SO}(6)$}
\author{Tomoyoshi Ibukiyama}

\address{Department of Mathematics, Graduate School of Science, Osaka University, Machikaneyama 1-1, Toyonaka, Osaka, 560-0043 Japan}

\email{ibukiyam@math.sci.osaka-u.ac.jp}

\author{Brandon Williams}

\address{Institut f\"ur Mathematik, Universit\"at Heidelberg, 69120 Heidelberg, Germany}

\email{bwilliams@mathi.uni-heidelberg.de}


\date{\today}

\maketitle

\begin{abstract}We state conjectures that relate Hermitian modular forms of degree two and algebraic modular forms for the compact group $\SO(6)$. We provide evidence for these conjectures in the form of dimension formulas and explicit computations of eigenforms.
\end{abstract}

\section{Introduction}

In this paper, we conjecture a correspondence between Hermitian 
modular forms of degree two and algebraic modular forms for the spin groups of certain quadratic forms in $6$ variables. \\

Whenever linear algebraic groups $G$ and $G'$ over $\Q$ satisfy
$G\otimes_{\Q} \C\cong G'\otimes_{\Q}\C$, the Langlands 
philosophy predicts the existence of a correspondence between 
automorphic forms for $G$ and $G'$ that is equivariant with respect to the actions of their Hecke algebras. The most familiar example of this sort of correspondence is between classical modular forms and certain harmonic polynomials in three or four variables, which can be regarded as algebraic modular forms for the group $\SU(2)$. This is a classical
result of Eichler \cite{eichler}, which was later generalized by Shimuzu \cite{shimizu} to the case of Hilbert modular forms and by Jacquet--Langlands \cite{jacquetlanglands} to the more general setting of automorphic 
representations of $GL(2)$.
For $\Sp(2,\R)$ and its compact twist $\Sp(2)$, 
there is also a correspondence of this type between paramodular forms and algebraic modular forms for $\Sp(2)$. This was first conjectured by Ibukiyama \cite{ibuparamodular}, \cite{ibuvect} and Ibukiyama--Kitayama \cite{ibukiyamakitayama},
and has now been proved by van Hoften 
\cite{hoften} and R\"{o}sner-Weissauer \cite{roesnerweissauer}.\\

Our paper is a continuation of this line of thought.
Our concern is not general automorphic representations; instead, we aim to give a correspondence between concretely described spaces of automorphic forms.
Since the compact twist of $\SU(2,2)$ is $\SU(4)$, it would be natural to
consider algebraic modular forms with respect to 
automorphism groups of hermitian lattices of rank $4$. Here we instead consider algebraic modular forms with respect to 
automorphism groups of rank six lattices, because we have the classical 
isogeny $\SU(4)/\{\pm 1_4\}\cong \SO(6)$ (such that $\SU(4)$ becomes the Spin group $\Spin(6)$). The other instances mentioned above have the similar descriptions $\SU(2)/\{\pm 1_2\}\cong \SO(3)$, 
$\SU(2)\times \SU(2)/\{\pm(1_2,1_2)\}\cong \SO(4)$, 
$\Sp(2)/\{\pm 1_2\}\cong \SO(5)$. \\

Algebraic modular forms for spin groups and orthogonal groups appear similar but they are actually quite different. For example, for certain ternary lattices the class numbers of $\mathrm{SO}(3)$ over $\Q$ are type numbers of quaternion algebras (i.e. numbers of isomorphism classes of maximal orders) while for $\Spin(3)$ they are class numbers of maximal orders.
Generally, if we denote by $\Gamma_e(Q)$ the 
even Clifford group of a quadratic form $Q$ over $\Q$, then 
we have $\Gamma_e(Q)/\Q^{\times}\cong \SO(Q)$, and 
if we denote by $\Spin(Q)$ the subgroup of elements in $\Gamma_e(Q)$ of 
spinor norm $1$, then the image of $\Spin(Q)$ in $\SO(Q)$ is much 
smaller than $\SO(Q)$ in general.
To pick out the spinor norm $1$ part of $\SO(Q)$, we have to work with spinor characters. This has been done for $\SO(n)$ with $n \le 5$, for example in \cite{dummigan}, \cite{RT}, \cite{tornaria}. \\

Algebraic modular forms for $\SO(6)$ that are associated to spherical representations can be described explicitly by harmonic polynomials that are invariant under the action of a certain group. In particular, these forms can always be computed directly. On the other hand, for small discriminants, the ring structure of Hermitian modular forms has been determined completely \cite{DK1}, \cite{DK2}, \cite{N}, \cite{W}. This makes it possible to compare dimensions and also to compare Hecke eigenvalues and Euler factors of $L$-functions at small primes. \\

We have carried out these comparisons experimentally and the object of this paper is to report on them. Hermitian modular forms for all discriminants contain a common Hecke algebra consisting of \emph{Miyawaki lifts} of pairs of cusp forms from $\mathrm{SL}_2(\mathbb{Z})$ that were constructed by Atobe-Kojima 
\cite{AK2018} using Ikeda \cite{Ikeda2008}. We conjecture that cuspidal Hermitian eigenforms that are \emph{not} Miyawaki lifts are in one-to-one correspondence with algebraic modular forms that are not \emph{Yoshida lifts}. Moreover, this correspondence respects $L$-functions. \\

The content of our paper is roughly 
as follows.
In Section 1, we review the definitions of Hermitian modular forms of degree two, the two types of lifts, their $L$-functions and the calculation of Hecke eigenvalues, the Atkin-Lehner theory, the volume of the fundamental domain, and the dimension formulas for spaces of modular forms in the few (small discriminant) known cases.
In Section 2, we review algebraic 
modular forms for $\SO(6)$, the spinor characters, and Hecke theory.
In Section 3, we describe exactly the modular groups to be considered and we make an explicit conjecture on the correspondence between Hermitian modular forms and algebraic modular forms for $\SO(6)$.
In Section 4, we calculate the Minkowski--Siegel volume of the genus of senary lattices we are considering and see that this matches the main term in the dimension formula for Hermitian modular forms. In Section 5, we compare the dimensions of Hermitian modular forms and  algebraic modular forms for small discriminants for all (scalar valued) weights, which can be viewed as evidence for our conjecture.
In Appendix A, we prove the dimension formula for Hermitian modular forms for $\mathbb{Q}(\sqrt{-2})$.
In Appendix B, we give tables of dimensions of the spaces Hermitian and algebraic modular forms, including their decomposition into lifts.
In Appendix C, we show that Hermitian modular forms and algebraic modular forms of low weight correspond for a larger range of imaginary quadratic fields. \\

\textbf{Acknowledgments.}
The first author was supported by JSPS KAKENHI Grant Number JP23K03031 and JP20H00115.
We thank Eran Assaf and John Voight for helpful discussions about algebraic modular forms, and Aloys Krieg for explaining the Hecke theory of Hermitian modular forms in great detail.

\section{Hermitian modular forms of degree two}

This section is meant to set up some background from the classical theory of modular forms for the group $\mathrm{SU}(2, 2)$ in terms of holomorphic functions on a tube domain, with emphasis on Hecke operators and eigenforms. None of the results here should be considered new.

\subsection{Hermitian modular forms}

For $n \in \mathbb{N}$, the \textbf{Hermitian upper half-space} of degree $n$ is the space of $(n \times n)$ complex matrices $Z$ with the property that, when we write $$Z = X + i Y \quad \text{with} \; X = \overline{X^T}, \; Y = \overline{Y^T},$$ the matrix $Y$ is positive-definite. Note that neither $X$ nor $Y$ is generally real. \\

The unitary group $\mathrm{SU}(n, n)$ will always consist of $(2n \times 2n)$ matrices $M$ for which $\mathrm{det}(M) = 1$ and $$M^T J \overline{M} = J,$$ where we fix the Hermitian form $J$ given by the block matrix $J = i \cdot \begin{pmatrix} 0 & -I \\ I & 0 \end{pmatrix}$.

For any order $\mathcal{O}$ in an imaginary-quadratic field $K$, let $\mathrm{SU}(n, n; \mathcal{O})$ be the subgroup of $\mathrm{SU}(n, n)$ of matrices with entries in $\mathcal{O}$.
In the case that $\mathcal{O}=\mathcal{O}_K$ is the maximal order, we simply write $\Gamma_K := \mathrm{SU}(n,n; \mathcal{O}_K)$.
The discriminant of $K$ is denoted by $\Delta_K$. Note that $\mathrm{SU}(n, n; \mathcal{O}_K) = \mathrm{U}(n, n; \mathcal{O}_K)$ unless $\Delta_K \in \{-3, -4\}$, see e.g. Remark 1.14 of \cite{Wernz2019}.

\begin{defn}
    A \textbf{Hermitian modular form} of weight $k \in 
    \mathbb{N}_0$ is a holomorphic function $f : \mathbf{H}_n \rightarrow \mathbb{C}$ with the property $$f\Big( (aZ+b)(cZ+d)^{-1} \Big) = \mathrm{det}(cZ+d)^k f(Z)$$ for every block matrix $\begin{pmatrix} a & b \\ c & d \end{pmatrix} \in \Gamma_K.$
\end{defn}
In the case $n = 1$, one additionally requires a growth condition at cusps. (This is essentially the same as modular forms for $\mathrm{SL}_2(\mathbb{Z})$.) We omit it because we are only interested in degree $n = 2$.

The invariance of a Hermitian modular form $f$ under translations $$Z \mapsto Z+b, \quad \text{where} \; b = \overline{b^T} \in \mathcal{O}_K^{n \times n}$$ implies that $f$ has a Fourier decomposition $$f(Z) = \sum_{c \in \Lambda_n} \alpha_f(c) e^{2\pi i \mathrm{tr}(cZ)},$$ where $\Lambda_n$ is the lattice $$\Lambda_n = \{\text{Hermitian matrices} \; c \; \text{such that} \; \mathrm{tr}(cb) \in \mathbb{Z} \; \text{for all} \; b \in \mathcal{O}_K^{n \times n}, \; b=\overline{b^T}\}.$$
In the special case of degree two, we usually write the Fourier series in the form $$f\Big( \begin{pmatrix} \tau & z_1 \\ z_2 & \omega \end{pmatrix} \Big) = \sum_{a, b, c} \alpha_f(a, b, c) q^a r^b s^c,$$ where $a, c \in \mathbb{N}_0$ (by the Koecher principle) and $b \in \mathcal{O}_K'$ \text{(=the inverse of the different)}, and where $$q = e^{2\pi i \tau}, \; s = e^{2\pi i \omega}, \; \text{and} \; r^b = e^{2\pi i (b z_1 + \overline{b} z_2)}.$$

The transpose defines an involution $Z \mapsto Z^T$ on $\mathbf{H}$. Modular forms that satisfy $$f(Z^T) = (-1)^k f(Z)$$ will be called symmetric, while modular forms that satisfy $$f(Z^T) = (-1)^{k+1} f(Z)$$ will be called skew-symmetric.

\subsection{Hecke operators}
Suppose $K = \mathbb{Q}(\sqrt{\Delta_K})$ where $\Delta_K < 0$ is a fundamental discriminant. We will describe the local Hecke algebras $\mathcal{H}_p$ at primes $p \nmid \Delta_K$. (See Proposition 4.6 and Proposition 4.7 of \cite{Klosin2014}, as well as \cite{HWK2020}.)

(1) Suppose $p$ is inert in $K$. Then the local Hecke algebra is generated by two double-coset operators $$T_p = \Gamma_K \begin{pmatrix} 1 & 0 & 0 & 0 \\ 0 & 1 & 0 & 0 \\ 0 & 0 & p & 0 \\ 0 & 0 & 0 & p \end{pmatrix} \Gamma_K \quad\text{and} \quad T_{p^2} = T(1, p, p^2, p) = \Gamma_K \begin{pmatrix} 1 & 0 & 0 & 0 \\ 0 & p & 0 & 0 \\ 0 & 0 & p^2 & 0 \\ 0 & 0 & 0 & p \end{pmatrix} \Gamma_K.$$

(2) Suppose $(p) = \mathfrak{p} \cdot \overline{\mathfrak{p}}$ splits in $K$ and that $\mathfrak{p} = (\pi)$ is principal. Then the local Hecke algebra has three generators, $T_p$, $T_{\mathfrak{p}}$, $T_{\overline{\mathfrak{p}}}$, where $$T_p = \Gamma_K \begin{pmatrix} 1 & 0 & 0 & 0 \\ 0 & 1 & 0 & 0 \\ 0 & 0 & p & 0 \\ 0 & 0 & 0 & p \end{pmatrix} \Gamma_K \quad \text{and} \quad T_{\mathfrak{p}} = T(1, \pi, p, \pi) = \Gamma_K \begin{pmatrix} 1 & 0 & 0 & 0 \\ 0 & \pi & 0 & 0 \\ 0 & 0 & p & 0 \\ 0 & 0 & 0 & \pi \end{pmatrix} \Gamma_K,$$ and where $T_{\overline{\mathfrak{p}}}$ is defined similarly to $T_{\mathfrak{p}}$ with $\pi$ replaced by $\overline{\pi}$. \\
(The classical setup of Hermitian modular forms is not the natural setting to discuss Hecke operators at non-principal ideals. However it may be worth mentioning that, by applying Chebotarev's density theorem to the Hilbert class field of $K$, it can be shown that there are infinitely many primes that split as $p = \pi \overline{\pi}$ with $\pi \in \mathcal{O}_K$.)\\

The $L$-function $L(f; s)$ of a Hermitian eigenform $f$ we need is the degree six zeta function, which is associated to the exterior square of the standard representation of $\mathrm{GL}_4$. The Euler factors of $L(f; s)$ were computed by Hina--Sugano \cite{HS1983}, Gritsenko \cite{G1986} and Sugano \cite{sugano} in terms of the Hecke eigenvalues of $f$. For the normalization of the Hecke operators used to define the eigenvalues below, see the next section. \\

Let $p$ be a good prime ($p \nmid \Delta_K$).

(1) Suppose $p$ splits as $p = \mathfrak{p} \overline{\mathfrak{p}}$. Let $\lambda_{\mathfrak{p}}, \lambda_{\overline{\mathfrak{p}}}$ be the eigenvalues of $f$ under $T_{\mathfrak{p}}$ ($= T(1, \pi, p, \pi)$) and  $T_{\overline{\mathfrak{p}}}$ and let $\lambda_p$ be the eigenvalue of $f$ under $T_p$. Then the Euler factor at $p$ is \begin{align*} L_p(f; s) &= 1 - \lambda_p X+ \Big( p^{k - 3} \lambda_{\mathfrak{p}} \lambda_{\overline{\mathfrak{p}}} - p^{2k - 4}\Big) X^2 - \Big( p^{2k - 5} (\lambda_{\mathfrak{p}}^2 + \lambda_{\overline{\mathfrak{p}}}^2) - 2p^{2k - 4} \lambda_p \Big) X^3 \\ &\quad + \Big( p^{3k - 7} \lambda_{\mathfrak{p}} \lambda_{\overline{\mathfrak{p}}} - p^{4k - 8}\Big) X^4 - p^{4k - 8} \lambda_p X^5 + p^{6k - 12} X^6, \end{align*} where $X = p^{-s}.$

(2) If $p$ is inert, then let $\lambda_p$ and $\lambda_{p^2}$ be the eigenvalues of $f$ under $T_p$ and $T_{p^2}$ ($=T(1,p,p^2,p)$). Then the Euler factor is $$L_p(f; s) = (1 - p^{2k - 4} X^2) \cdot \Big[ 1 - \lambda_p X + \Big(p^{k-3} \lambda_{p^2} + p^{2k - 7} (p^3 + p^2 - p + 1) \Big) X^2 - p^{2k - 4} \lambda_p X^3 + p^{4k - 8} X^4 \Big],$$ where again $X = p^{-s}.$

The Euler factors at ramified primes will not be used. For $K = \mathbb{Q}(i)$ (and therefore $p=2$) they were given by Gritsenko \cite{G1986}, and for general $K$ they were computed by Sugano \cite{sugano}.\\

The completed $L$-function was defined in \cite{G1986} for $K = \mathbb{Q}(i)$, and the definition over other fields is analogous. If $f$ has weight $k$, then the completed $L$-function is
$$\Lambda(f; s) = (2\pi)^{-3s} |\Delta_K|^{s/2} \Gamma(s) \Gamma(s-k+1) \Gamma(s-k+2)L(f; s),$$
and its functional equation relates the values at $s$ and $2k-3-s$.

\begin{ex}
In our convention, the Hermitian-Siegel Eisenstein series has zeta function $$L(E_k; s) = L(\chi_K; s-k+2) \zeta(s) \zeta(s - k + 1) \zeta(s-k+2) \zeta(s-k+3) \zeta(s - 2k+4).$$
\end{ex}

\subsection{Formulas for Hecke operators}

All of the double cosets mentioned in the previous section have right-coset decompositions which are represented by block upper-triangular matrices. Here we list (without proof) such representatives. See also \cite{HWK2020}, \cite{Klosin2014}.

The right-coset decompositions yield formulas for the action of Hecke operators on Fourier coefficients, which we used to decompose spaces of Hermitian modular forms into eigenspaces and to compute the degree six $L$-functions of eigenforms. An implementation of these formulas in SageMath has been made available \cite{WeilRep}. \\

\subsubsection{Inert primes} Suppose $p$ is inert in $K$. The right cosets contained in $T_p$ have block upper-triangular representatives of the following four types:
\begin{enumerate}

\item $\displaystyle \begin{pmatrix} I & B \\ 0 & pI \end{pmatrix}, \quad \text{where} \; B \in (\mathcal{O}_K/ p)^{2 \times 2} \; \text{with} \; B = \overline{B^T};$

\item $\displaystyle\begin{pmatrix} pI & 0 \\ 0 & I \end{pmatrix};$

\item $\displaystyle \begin{pmatrix} p & 0 & 0 & 0 \\ -\overline{d} & 1 & 0 & \beta \\ 0 & 0 & 1 & d \\ 0 & 0 & 0 & p \end{pmatrix}, \quad \text{where} \; d \in \mathcal{O}_K / p \; \text{and} \; \beta \in \mathbb{Z}/p;$

\item $\displaystyle M_{\beta, p} = \begin{pmatrix} 1 & 0 & \beta & 0 \\ 0 & p & 0 & 0 \\ 0 & 0 & p & 0 \\ 0 & 0 & 0 & 1 \end{pmatrix}, \quad \beta \in \mathbb{Z}/p.$
\end{enumerate}

This determines the action of $T_p$ on the Fourier coefficients of modular forms. In order to make $T_p$ preserve integrality of Fourier coefficients, we use the normalization $$T_p f(Z) = p^{2k - 4} \sum_M \mathrm{det}(cZ+d)^{-k} f((aZ+b)(cZ+d)^{-1})$$ where $M = \begin{pmatrix} a & b \\ c & d \end{pmatrix}$ runs through the representatives defined above. The result is that if $$f(Z) = \sum_{a,b,c} \alpha_f(a, b, c) q^a r^b s^c,$$ then $T_p f(Z) = \sum_{a, b, c} \beta(a,b,c) q^a r^b s^c$ where \begin{align*} \beta(a, b, c) &= \alpha_f(pa, pb, pc) \\ &+ p^{2k-4} \alpha_f(a/p, b/p, c/p) \\ &+ p^{k-3} \alpha_f(pa, b, c/p) + p^{k-3} \sum_{d \in \mathcal{O}_K / p} \alpha_f \Big( \frac{a + db + \overline{db} + d\overline{d}c}{p}, b + \overline{d}c, pc \Big). \end{align*}

For $T_{p^2} = T(1,p,p^2,p)$, we have block upper-triangular right coset representatives of the following six types: 

\begin{enumerate}

\item $\displaystyle \begin{pmatrix} p^2 & 0 & 0 & 0 \\ 0 & p & 0 & 0 \\ 0 & 0 & 1 & 0 \\ 0 & 0 & 0 & p \end{pmatrix};$

\item $\displaystyle\begin{pmatrix} p & p \alpha & 0 & 0 \\ 0 & p^2 & 0 & 0 \\ 0 & 0 & p & 0 \\ 0 & 0 & -\overline{\alpha} & 1 \end{pmatrix}, \quad \alpha \in \mathcal{O}_K / p;$ 

\item $\displaystyle\begin{pmatrix} 1 & \alpha & \beta+\alpha\overline{\gamma} & \gamma \\ 0 & p & \overline{\gamma} p & 0 \\ 0 & 0 & p^2 & 0 \\ 0 & 0 & -\overline{\alpha} p & p \end{pmatrix}, \quad \text{where} \; \alpha, \gamma \in \mathcal{O}_K / p \; \text{and} \; \beta \in \mathbb{Z}/(p^2);$

\item $\displaystyle \begin{pmatrix} p & 0 & 0 & \gamma p \\ 0 & 1 & \overline{\gamma} & \beta \\ 0 & 0 & p & 0 \\ 0 & 0 & 0 & p^2 \end{pmatrix}, \quad \text{where} \; \gamma \in \mathcal{O}_K / p \; \text{and} \; \beta \in \mathbb{Z}/(p^2);$

\item $\displaystyle \begin{pmatrix} p & 0 & b & 0 \\ 0 & p & 0 & d \\ 0 & 0 & p & 0 \\ 0 & 0 & 0 & p \end{pmatrix}, \quad \text{where} \; b, d \in \mathbb{Z}/p \; \text{with} \; bd=0 \; \text{but} \; (b,d)\ne(0,0);$

\item $\displaystyle \begin{pmatrix} p & 0 & \beta & \gamma \\ 0 & p & \overline{\gamma} & \beta^* |\gamma|^2 \\ 0 & 0 & p & 0 \\ 0 & 0 & 0 & p \end{pmatrix}, \quad \beta \in (\mathbb{Z}/p\mathbb{Z})^{\times}, \; \gamma \in (\mathcal{O}_K / p)^{\times},$ where $\beta^* \beta \equiv 1 \, (p).$
\end{enumerate}

The normalization of the $T_{p^2}$ operator is $$T_{p^2} f(Z) = p^{3k-4} \sum_M \mathrm{det}(cZ + d)^{-k} f((aZ+b)(cZ+d)^{-1}),$$ where $M$ runs through the six coset types listed above. If $f(Z) = \sum_{a, b, c} \alpha_f(a, b, c) q^a r^b s^c$ then we obtain $$T_{p^2} f(Z) = \sum_{a,b,c} \beta(a,b,c) q^a r^b s^c$$ with coefficients
\begin{align*} \beta(a,b,c) &= p^{2k-4} \alpha_f(a / p^2, b/p, c) \\ &+ p^{2k-4} \sum_{\alpha \in \mathcal{O}_K / p} \alpha_f \Big( a, (b - a\overline{\alpha})/p, (c - p\alpha b - p \overline{\alpha b} + a |\alpha|^2) / p^2 \Big) \\ &+\alpha_f(a, pb, p^2 c) \\ &+ \sum_{\alpha \in \mathcal{O}_K / p} \alpha_f \Big( p^2 a, p(b - a \overline{\alpha}), c - b \alpha - \overline{b \alpha} + a |\alpha|^2 \Big) \\ &+ p^{k-4} \Big( \epsilon_p(a, c) + \nu_p(a,b,c) \Big) \alpha_f(a, b, c), \end{align*} where the constants $\epsilon_p$ and $\nu_p$ are defined by $$\epsilon_p(a, c) := \begin{cases} 2p-2: & a \equiv c \equiv 0 \, (\text{mod}\, p); \\ p-2: & a \equiv 0 \; \text{and} \; c \not\equiv 0 \, (\text{mod}\, p); \\ p-2: & a \not \equiv 0 \; \text{and} \; c \equiv 0 \, (\text{mod}\, p); \\ -2: & \text{otherwise}; \end{cases}$$ and $$\nu_p(a, b, c) := 1-p^2 + p \cdot \# \Big\{ \gamma \in (\mathcal{O}_K / p)^{\times}: \; a + \gamma b + \overline{\gamma b} + c |\gamma|^2 \equiv 0 \, (\text{mod}\, p) \Big\}.$$

\subsubsection{Split primes} Suppose $(p) = \mathfrak{p} \mathfrak{\overline{p}}$ is a split prime in $K$ with $p \nmid \Delta_K$ and where $\mathfrak{p} = (\pi)$ is principal. We have the following right-coset representatives for $T_{\mathfrak{p}} = T(1, \pi, p, \pi)$:

\begin{enumerate}
\item $\displaystyle \begin{pmatrix} 1 & \alpha & \beta & \gamma \\ 0 & \pi & \pi \overline{\gamma} & 0 \\ 0 & 0 & p & 0  \\ 0 & 0 & -\pi \overline{\alpha} & \pi \end{pmatrix}, \quad \alpha, \gamma \in \mathcal{O}_K / \pi, \; \beta \in \mathbb{Z}/p,$ where the representatives $\alpha, \gamma \in \mathcal{O}_K / \pi$ are chosen such that $\alpha \overline{\gamma} \in \mathbb{Z}$; 

\item $\displaystyle \begin{pmatrix} \pi & 0 & 0 & \pi \overline{\gamma} \\ 0 & 1 & \gamma & \beta \\ 0 & 0 & \pi & 0 \\ 0 & 0 & 0 & p \end{pmatrix}, \quad \gamma \in \mathcal{O}_K / \pi, \; \beta \in \mathbb{Z}/p;$ 

\item $\displaystyle \begin{pmatrix} p & 0 & 0 & 0 \\ -\pi \overline{\alpha} & \pi & 0 & 0 \\ 0 & 0 & 1 & \alpha \\ 0 & 0 & 0 & \pi \end{pmatrix}, \quad \alpha \in \mathcal{O}_K / \pi;$

\item $\displaystyle \begin{pmatrix} \pi & 0 & 0 & 0 \\ 0 & p & 0 & 0 \\ 0 & 0 & \pi & 0 \\ 0 & 0 & 0 & 1 \end{pmatrix}.$ 

\end{enumerate}

We normalize the action on modular forms as $$T_{\mathfrak{p}} f(Z) = \pi^k p^{k-2} \sum_M \mathrm{det}(cZ+d)^{-k} f((aZ+b)(cZ+d)^{-1}).$$ 
Then the action on Fourier coefficients is as follows: if $f(Z) = \sum_{a, b, c} \alpha_f(a, b, c) q^a r^b s^c$, then $$T_{\mathfrak{p}} f(Z) = \sum_{a, b, c} \beta(a, b, c) q^a r^b s^c$$ where \begin{align*} \beta(a, b, c) &= \alpha_f(a, \overline{\pi} b, pc) \\ &+ \sum_{\alpha \in \mathcal{O}_K / \pi} \alpha_f \Big( pa, \pi(b - a \overline{\alpha}), c - \alpha b - \overline{\alpha b} + a |\alpha|^2 \Big) \\ &+ p^{k-2} \alpha_f(a, b/\pi, c/p) \\ &+ p^{k-2} \sum_{\alpha \in \mathcal{O}_K / \pi} \alpha_f \Big( \frac{a + \alpha b \pi + \overline{\alpha b \pi} - c |\alpha|^2}{p}, \frac{b + c \overline{\alpha}}{\overline{\pi}}, c \Big). \end{align*}

For the $T_p$-operator, we have the same four types of right coset representatives as in the case of inert primes, but also two additional types of right cosets:

\begin{enumerate}
\setcounter{enumi}{4}
\item $\displaystyle \begin{pmatrix} 1 & 0 & \beta & 0 \\ 0 & p & 0 & 0 \\ 0 & 0 & p & 0 \\ 0 & 0 & 0 & 1 \end{pmatrix} \begin{pmatrix} 1 & \overline{d} & 0 & 0 \\ 0 & 1 & 0 & 0 \\ 0 & 0 & 1 & 0 \\ 0 & 0 & d & 1 \end{pmatrix}$;
\item $\displaystyle \begin{pmatrix} 1 & 0 & \beta & 0 \\ 0 & p & 0 & 0 \\ 0 & 0 & p & 0 \\ 0 & 0 & 0 & 1 \end{pmatrix} \begin{pmatrix} 1 & \overline{d} & 0 & 0 \\ 1 & \overline{d} + 1 & 0 & 0 \\ 0 & 0 & 1 & 1 \\ 0 & 0 & d & d+1 \end{pmatrix}$.
\end{enumerate}

In both (5) and (6), $\beta$ runs through $\mathbb{Z}/p\mathbb{Z}$ and $d$ represents certain classes in $\mathcal{O}_K / p$: In (5), we require $d \overline{d} \in p\mathbb{Z}$ but $d \notin p \mathcal{O}_K$.
In (6), we require both $d \overline{d} \in p\mathbb{Z}$ and $d+\overline{d} \in p\mathbb{Z}$ but $d \notin p \mathcal{O}_K$.

The $T_p$-operator has the same normalization factor $p^{2k-4}$ that it does when $p$ is inert. If $$f(Z) = \sum_{a,b,c} \alpha_f(a, b, c) q^a r^b s^c,$$ then $T_p f(Z) = \sum_{a, b, c} \beta(a,b,c) q^a r^b s^c$ where 

\begin{align*} \beta(a, b, c) &= \alpha_f(pa, pb, pc) \\ &+ p^{2k-4} \alpha_f(a/p, b/p, c/p) \\ &+ p^{k-3} \alpha_f(pa, b, c/p) + p^{k-3} \sum_{d \in \mathcal{O}_K / p} \alpha_f \Big( \frac{a + db + \overline{db} + d\overline{d}c}{p}, b + \overline{d}c, pc \Big) \\ &+ p^{k-3} \sum_{\substack{d \in \mathcal{O}_K / p \\ d \notin p \mathcal{O}_K \\ d \overline{d} = 0 \, \text{mod}\, p}} \alpha_f \Big( pa, b + da, \frac{c + b\overline{d} + \overline{b} d + d \overline{d}a}{p} \Big) \\ &+ p^{k-3} \sum_{\substack{d \in \mathcal{O}_K / p \\ d \notin p \mathcal{O}_K \\ d \overline{d} = 0 \, \text{mod}\, p \\ d + \overline{d} = 0 \, \text{mod}\, p}} \alpha_f(a', b', c'), \end{align*} where in the last line we set \begin{align*} a' &= a + b + \overline{b} + c, \\ b' &=  d(a+b+\overline{b}+c)+b+c, \\ c' &= \frac{c(1+d)(1+\overline{d}) + b \overline{d} + \overline{b}d + (a+b+\overline{b})d \overline{d}}{p}.\end{align*}

\subsection{Liftings of elliptic modular forms}

The Hermitian modular group of degree two admits two kinds of liftings from elliptic modular forms. \\

The first is the Gritsenko--Maass lift. This is the degree $n=2$ case of the general Ikeda lift for Hermitian modular forms \cite{Ikeda2008}.

The Gritsenko--Maass lift takes Jacobi cusp forms $f$ of weight $k$ whose index is the lattice of integers $\mathcal{O}_K$, with quadratic form given by the norm $N_{K/\mathbb{Q}}$, and lifts them to Hermitian modular forms $M_f$ of the same weight. The span of the image is called the \textbf{Sugano Maass space} $$\mathrm{Maass}_k(\Gamma_K) \subseteq S_k(\Gamma_K).$$
The Maass lift maps Hecke eigenforms to Hecke eigenforms. The Euler factors of the degree six zeta function were computed by Gritsenko \cite{Gritsenko1990} in the case of the field $\mathbb{Q}(i)$, but the proof carries over to any imaginary-quadratic field, and the Euler factors in more generality were obtained by Sugano (\cite{Sugano1995}, Theorem 8.1).

Hecke operators on Jacobi forms of lattice index are defined similarly to scalar index; for details, we refer to \cite{Ajouz} or
\cite{Sugano1995}. For a prime $p$, the operator $T_p$ is simply the sum of the action of the double coset of $\begin{pmatrix} p^{-1} & 0 \\ 0 & p \end{pmatrix}$ and multiplication by $p^{k-3} \chi_K(p)$. If $f$ is a Jacobi eigenform of weight $k$ and lattice index $\mathcal{O}_K$, with eigenvalues $\lambda_p$ under the operators $T_p$ as defined in \cite{Ajouz}, 
then the Euler factor of $L(M_f; s)$ at a prime $p \nmid |\Delta_K|$  is \begin{align*} &\quad L_p(M_f; s) \\ &= (1 - p^{k-s-3})(1 - p^{k-s-2})(1 - p^{k-s-1}) \cdot \Big( 1 - \lambda_p p^{-s} + \chi_K(p) \lambda_p p^{k-2-2s} - \chi_K(p) p^{3k-6-3s} \Big).\end{align*}

Note that if $k$ is even and $q = -\Delta_K$ is a prime, then $f$ corresponds naturally to a pair of conjugate modular forms $$g, \overline{g} \in S_{k-1}(\Gamma_0(q), \chi_K)$$ under the map defined in \cite{BB2003}, and that $g, \overline{g}$ are eigenforms of all Hecke operators $T_n$ with $\left( \frac{n}{q} \right) = 1$ and their $p^2$-eigenvalues are $\lambda_p$. The factors $$1 - \lambda_p p^{-s} + \chi_K(p) \lambda_p p^{k-2-2s} - \chi_K(p) p^{3k-6-3s}$$ in $L_p(M_f; s)$ are exactly the Euler factors in the symmetric square $L$-function attached to $g$ by Theorem 5.1.6 of \cite{Ajouz}. Hence (up to primes dividing $\Delta_K$) we have $$L(M_f; s) = \zeta(s-k+1) \zeta(s-k+2) \zeta(s-k+3) L(\mathrm{Sym}^2 g; s).$$ Jacobi forms of odd weight have a somewhat more complicated interpretation in terms of modular forms of level $q^2$ instead of $q$ \cite{SW2019}. \\

The second type of lifting is the Miyawaki lift, which is obtained by pulling back an Ikeda lift and integrating against another modular form. The Miyawaki lift for Hermitian modular forms was defined in general in \cite{AK2018}.

In the case of Hermitian modular forms of degree two, the Miyawaki lift begins with two cuspidal eigenforms $f$ and $g$ for $\mathrm{SL}_2(\mathbb{Z})$ of weights $k$ and $k + 2$, respectively. By Ikeda's construction \cite{Ikeda2008}, the form $f$ lifts to a Hermitian modular form $F$ of degree $3$ and weight $k + 2$. Then the Miyawaki lift is defined by the integral $$\mathcal{F}_{f, g}(Z) := \int_{\mathrm{SL}_2(\mathbb{Z}) \backslash \mathbb{H}} F \Big( \begin{pmatrix} Z & 0 \\ 0 & w \end{pmatrix} \Big) 
 \overline{g(w)} v^k \, \frac{\mathrm{du} \, \mathrm{d}v}{v^2}, \quad w = u+iv.$$

It was shown in \cite{AK2018} that $\mathcal{F}_{f, g}$ is an eigenform of the Hecke operators on $\U(2, 2)$, and its standard $L$-function is computed there as well. The degree six $L$-function is not given explicitly in \cite{AK2018}, but from numerical examples it is  clear that the degree six $L$-function is $$L(\mathcal{F}_{f, g}; s) = \zeta_K(s - k + 2) \cdot L(f \otimes g; s),$$ and presumably this can be derived from the work of \cite{AK2018}. Here $\zeta_K$ is the Dedekind zeta function and $L(f \otimes g; s)$ is the Rankin--Selberg product $L$-function. 
In other words, if $p \nmid \Delta_K$ and if the Euler factors of $L(f; s)$ and $L(g; s)$ at $p$ are factored $$L_p(f; s)^{-1} = (1 - \alpha_f(p) p^{-s})(1 - \beta_f(p) p^{-s}), \quad L_p(g; s)^{-1} = (1 - \alpha_g(p) p^{-s})(1 - \beta_g(p) p^{-s}),$$ then the Euler factor of $L(
\mathcal{F}_{f,g}; s)$ at $p$ is 
 \begin{align*} L_p(
 \mathcal{F}_{f,g}; s)^{-1} &= (1 - p^{-s+k-2}) (1 - \chi_K(p) p^{-s+k-2}) \\ & \times (1 - \alpha_f(p) \alpha_g(p) p^{-s}) (1 - \alpha_f(p) \beta_g(p) p^{-s}) (1 - \beta_f(p) \alpha_g(p) p^{-s}) (1 - \beta_f(p) \beta_g(p) p^{-s}). \end{align*}

\subsection{Atkin--Lehner theory}\label{subs:AL}

The maximal discrete extension of $\Gamma_K$ was described by Krieg, Raum and Wernz \cite{kriegraumwernz} in terms of Atkin--Lehner involutions $W_p$. 

Let $d$ be a divisor of $\Delta_K$ with $4 \nmid d$, and set $m = -\Delta_K$ (if $\Delta_K$ is odd) or $m = -\Delta_K / 4$ (if $\Delta_K$ is even). Then $d$ and $\frac{m(m+1)}{d}$ are coprime integers and we choose $u, v \in \mathbb{Z}$ such that $ud - v \frac{m(m+1)}{d} = 1$. With $$V_d := \frac{1}{\sqrt{d}} \begin{pmatrix} ud & v (m + \sqrt{-m}) \\ m - \sqrt{-m} & d \end{pmatrix} \in \mathrm{SL}_2(\mathbb{C})$$ the Atkin--Lehner involution is defined as the matrix $$W_d := \begin{pmatrix} \overline{V_d^T} & 0 \\ 0 & V_d^{-1} \end{pmatrix} \in \mathrm{SU}(2, 2; \mathbb{C}).$$ This is independent of $u, v$ up to multiplication by $\Gamma_K$.

It is shown in \cite{kriegraumwernz} that the operators $W_d$ normalize $\Gamma_K$ and generate the maximal discrete extension $$\Gamma_K^* = \langle \Gamma_K, W_d: \, 4 \nmid d | \Delta_K \rangle,$$ and that $\Gamma_K^* / \Gamma_K \cong (\mathbb{Z}/2\mathbb{Z})^{t}$ where $t$ is the number of prime divisors of $\Delta_K$. \\

The Atkin--Lehner involutions commute with all Hecke operators and therefore any Hermitian eigenform is also an eigenform of all $W_d$ with eigenvalue $\pm 1$. Wernz \cite{Wernz2020} proved that the subspace of the Sugano Maass space of forms with sign $+1$ under all Atkin--Lehner involutions is exactly the Maass space that was defined by Krieg \cite{K1991} in terms of relations satisfied by Fourier coefficients.

\subsection{Dimensions}

The dimensions of modular forms for the full modular group $\Gamma_K$ are known only in a few cases. But one can approximate the dimension in terms of the volume of a fundamental domain for $\Gamma_K$ (cf. \cite{Freh}, Hauptsatz II.3.2). If $K$ has discriminant $\Delta_K$ and character $\chi_K = n \mapsto \left( \frac{\Delta_K}{n}\right)$, then a fundamental domain for $\mathrm{U}(2, 2; \mathcal{O}_K)$ has volume

\begin{align}\label{eq:volume} \mathrm{vol}(\mathcal{F}) &= \frac{1}{4\pi^5}|\Delta_K|^{5/2} \zeta(2) L(3, \chi_K) \zeta(4) \nonumber \\ &= \frac{\pi}{2160} |\Delta_K|^{5/2} L(3, \chi_K) \nonumber \\ &= \frac{\pi^4}{3240} B_{3, \chi_K}. \end{align} 
In the last line, $B_{n,\chi_K}$ is a Bernoulli number defined (cf. \cite{bernoulli}) by the generating function 
\[
\sum_{a=1}^{D}\frac{\chi(a)te^{at}}{e^{Dt}-1}=\sum_{n=1}^{\infty}B_{n,\chi_K}\frac{t^n}{n!}.
\]
In particular,
\[
B_{3,\chi_K}=\frac{1}{D}\sum_{a=1}^{D}\chi(a)a^3
-\frac{3}{2}\sum_{a=1}^{D}\chi(a)a^2+\frac{D}{2}
\sum_{a=1}^{D}\chi(a)a
\] where $D = -\Delta_K$.

\begin{table}[ht]
\caption{Quadratic Bernoulli numbers $B_{3, \chi_K}$}
\begin{tabular}{|l|l|l|l|l|l|l|l|l|l|l|l|l|}
\hline
$|\Delta_K|$      & $3$   & $4$   & $7$    & $8$ & $11$ & $15$ & $19$ & $20$ & $23$  & $24$ & $31$ & $35$ \\ \hline
$B_{3, \chi_K}$ & $2/3$ & $3/2$ & $48/7$ & $9$ & $18$ & $48$ & $66$ & $90$ & $144$ & $138$ & $288$ & $324$ \\ \hline
\end{tabular}
\end{table}

This volume formula can equivalently be derived from the well-known isogeny from $\SU(2,2)$ to $\SO(2, 4)$ and the formula of \cite{GHS2007} for the Hirzebruch--Mumford volume of $\O(2, n)$. \\

If $\Delta_K \ne -4$, then the above formula gives the volume of a fundamental domain for $\Gamma_K$. (If $\Delta_K \ne -3, -4$, then $\mathrm{U}(2, 2;\mathcal{O}_K)$ and $\Gamma_K$ coincide; while if $\Delta_K = -3$, then $\mathrm{U}(2, 2;\mathcal{O}_K)$ is given by extending $\Gamma_K$ by the central elements $\omega I$, $\omega \in \{e^{\pi i /3}, e^{2\pi i / 3}\}$ which act trivially on $\mathbf{H}$, so the fundamental domains are the same in this case also.) When $\Delta_K = -4$, we have $[U(2,2;O_K):\Gamma_K]=2$ and the volume formula \eqref{eq:volume} must be multiplied by two. \\

The contributions of the center of $\Gamma_K$ to the Selberg trace formula determine the asymptotic growth of $\mathrm{dim}\, M_k(\Gamma_K)$ as $k$ becomes large: \begin{equation} \label{eq:asymptotic} \mathrm{dim}\, M_k(\Gamma_K) \sim \frac{k^4}{64\pi^4} \cdot \mathrm{vol}(\mathcal{F}) \sim k^4 \cdot \frac{B_{3, \chi_K}}{207360}.\end{equation}
In the exceptional case $K = \mathbb{Q}(i)$, we have to multiply \eqref{eq:asymptotic} by two, and for even $k$ we obtain $$\mathrm{dim}\, M_k(\Gamma_{\mathbb{Q}(i)}) \sim k^4 \cdot \frac{2 B_{3, \chi_{\mathbb{Q}(i)}}}{207360} = \frac{k^4}{69120}$$ while for odd $k$ there are no modular forms at all.\\

There are a few cases where the dimensions are known completely because the structure of the underlying graded algebra of modular forms has been worked out. We summarize these cases below:

\begin{thrm} The dimensions of Hermitian modular forms of degree two for discriminants $\Delta = -3, -4, -7, -8, -11$ have the following generating series. \\
(i) For $\Delta = -3$, $$\sum_{k=0}^{\infty} \mathrm{dim} \, M_k(\Gamma_K) X^k = \frac{1+X^{45}}{(1 - X^4)(1 - X^6)(1 - X^9)(1 - X^{10})(1 - X^{12})}.$$

(ii) For $\Delta = -4$, $$\sum_{k=0}^{\infty} \mathrm{dim} \, M_k(\Gamma_K) X^k = \frac{(1+X^{10})(1+X^{34})}{(1 - X^4)(1 - X^6)(1 - X^8)(1 - X^{10})(1 - X^{12})}.$$

(iii) When $\Delta = -7$,
\[
\sum_{k=0}^{\infty} \mathrm{dim}\, M_k(\Gamma_K) X^k = \frac{P(X)}{(1 - X^4)(1  - X^6)(1 - X^{10})(1 -X^{12})(1-X^{14})}
\]
with the polynomial 

\begin{align*} P(X) &= 1 + X^8 + X^{10} + 2X^{16} + 2X^{18} + X^{24} + X^{26} + 2X^{32} + 2X^{34} + X^{40} + X^{42} + X^{50} \\ &+ X^7 + X^9 + X^{11} + X^{15} + X^{17} + X^{19} + X^{23} + X^{25} + X^{27} + X^{31} + X^{33} + 2X^{35} \\ &\quad\quad + 2X^{39} + X^{41} + X^{43} - X^{49}.
\end{align*}

(iv) When $\Delta= -8$, 

$$\sum_{k=0}^{\infty} \mathrm{dim}\, M_k(\Gamma_K) X^k = \frac{P(X)}{(1 - X^2)(1 - X^6)(1 - X^8)(1 - X^{10})(1 - X^{12})}$$ with the polynomial \begin{align*}P(X) &= 1 - X^2 + X^4 + X^8 + X^{12} + X^{30} + X^{34} + X^{38} - X^{40} + X^{42} \\ & + X^9 + X^{15} + X^{23} - X^{25} + 2X^{27} - X^{29} + X^{31} + X^{33}  + X^{35} - X^{37} + X^{39} - X^{41}.\end{align*}

(v) When $\Delta = -11$, $$\sum_{k=0}^{\infty} \mathrm{dim}\, M_k(\Gamma_K) X^k = \frac{P(X)}{(1 - X^4)(1 - X^6)^2 (1 - X^{10})(1 - X^{12})}$$ with the polynomial 
\begin{align*} P(X) &= 1 + 2X^8 + 2X^{10} + X^{12} + X^{16} + X^{18} + X^{24} + X^{26} + X^{28} + X^{30} + 2X^{32} + 2X^{34} + X^{42} \\ &+ X^5 + X^7 + 2X^9 + X^{11} + X^{13} + X^{15} + X^{17} + X^{19} + X^{23} + X^{27} + X^{29} \\ &\quad\quad + 2X^{31} + 2X^{33} + 2X^{35} + X^{37} - X^{41}.
\end{align*}

\end{thrm}
\begin{proof}
The result for $D = -3$ was proved by Dern and Krieg \cite{DK1}. For $D = -4$, the dimensions of symmetric modular forms were described by Nagaoka \cite{N} and Ibukiyama \cite{I} and antisymmetric modular forms by Aoki \cite{A}; see also \cite{DK1}. For discriminants $D = -7$ and $D = -11$ the dimensions were calculated by Williams \cite{W}. Finally, for $D = -8$, Dern and Krieg \cite{DK2} determined generators for both symmetric modular forms, and skew-symmetric modular forms over it, but they do not obtain the ideal of relations or the dimensions. A dimension formula in this case is derived in Appendix A below using Jacobi forms. \qedhere
\end{proof}

\begin{rem}
 In a power series of the form $$f(t) = \sum_{k=0}^{\infty} c_k t^k = \frac{P(t)}{(1 - t^{a_1})(1 - t^{a_2})(1 - t^{a_3})(1 - t^{a_4})(1 - t^{a_5})}$$ where $a_1,...,a_5 \in \mathbb{N}$ are coprime, the growth of the coefficients $c_k$ is dominated by the behavior of $f$ near $t=1$: the partial fractions expansion has the form 
 $$f(t) = \frac{p_0(t)}{(1 - t)^5} + \sum_i \frac{p_i(t)}{(1 - \zeta_i t)^{n_i}}$$ for certain polynomials $p_i(t)$, where $\zeta_i$ are roots of unity and $n_i \le 4$. Considering the limit $\lim_{t\rightarrow 1}(1-t)^5f(t)$ shows that $p_0(1)=P(1)/a_1a_2a_3a_4a_5$. 
 Then apply the geometric series: the terms $\frac{p_i(t)}{(1 - \zeta_i t)^{n_i}}$ are negligible, and $$c_k \sim \frac{P(1)}{a_1 a_2 a_3 a_4 a_5} \binom{k+4}{4} \sim \frac{P(1)}{24 a_1 a_2 a_3 a_4 a_5} k^4, \quad k \rightarrow \infty.$$

  This is a quick sanity check for the main term in the dimension formula in any case where the full Hilbert series is known, and it can be applied to the Hilbert series indicated above. For example, with $P(X) = 1+X^{45}$ and $(a_1,...,a_5) = (4,6,9,10,12)$ we obtain $$\mathrm{dim}\, M_k(\Gamma_{\mathbb{Q}(\sqrt{-3})}) \sim \frac{k^4}{311040} = k^4 \cdot \frac{2/3}{207360}.$$ With $P(X) = (1+X^5)(1 + X^{17})$ and $(a_1,...,a_5) = (2,3,4,5,6)$ we obtain $$\mathrm{dim}\, M_{2k}(\Gamma_{\mathbb{Q}(i)}) \sim \frac{k^4}{4320} = \frac{(2k)^4}{69120}.$$
  \end{rem}

\section{Algebraic modular forms}

\subsection{Algebraic modular forms on \texorpdfstring{$\SO$}{SO}}
We recall the definition of algebraic modular forms with respect to the 
special orthogonal group.
Let $V$ be a finite-dimensional 
vector space over $\Q$ with a 
positive-definite quadratic form $Q$. 
The algebraic group $\SO(V)$ is defined over 
$\Q$ by
\[
\SO(V)=\{g\in \SL(V): Q(gv)=Q(v) \text{ for all }v\in V\}.
\]
The adelization $\SO(V)_A$ and 
$v$-component $\SO(V)_v$ for any place $v$ of $\Q$ are defined
as usual. Here $\SO(V)_{\infty}$ is the compact 
orthogonal group $\SO(m)$ where $m=\dim V$. \\

For simplicity, we assume from now on that $m\geq 3$.

Fix a finite-dimensional 
rational irreducible representation 
$(\rho,W)$ of $\SO(m)$ and extend it to a representation of $\SO(V)_A$ by 
$$\rho:\SO(V)_A\rightarrow \SO(V)_{\infty}=\SO(m)
\rightarrow \GL(W).$$
In addition, fix a lattice $L$ of $V$ 
(i.e. a free $\Z$-module containing a basis of $V$ 
over $\Q$).
For a finite place $v$ of $\Q$, define 
$L_v=L\otimes_{\Z}\Z_v$. 
We let $\SO(V)$ act on $V$ from the right. For 
$g = (g_v)_v \in \SO(V)_A$, we define a lattice $Lg$ of $V$ 
by 
\[
Lg=\cap_{v<\infty}(L_vg_v\cap V).
\]
and we define a subgroup of $\SO(V)_A$ by 
\[
\U(L)=\{g\in \SO(V)_A: Lg=L\}.
\]
Classically, automorphic forms with respect to 
$\SU(2)$ (or $SO(3)$) are called Brandt matrices and are used for comparison between elliptic modular forms and certain harmonic polynomials. 
The more general definition for algebraic groups whose 
infinite part is compact up to the center was 
first given by Ihara \cite{ihara} in the symplectic case, and the basic theory of Hecke operators and the 
trace formula was given by Hashimoto \cite{hashimoto}.
The more general formulation is due to Gross \cite{gross}, who also introduced the name ``algebraic modular forms".

By definition, algebraic modular forms 
for $\SO(V)$ of weight $\rho$ 
with respect to $U(L)$ are elements of 
the following space:
\begin{align*}
{\mathfrak M}_{\rho}(L)=\{&f:\SO(V)_A\rightarrow W: \\ &
f(uga)=\rho(u)f(g) \text{ for any }u\in \U(L), g\in \SO(V)_A, a\in \SO(V)\}.
\end{align*}
This can be written more concretely as follows.
Let $g_i$ be representatives of the double coset decomposition
\[
\SO(V)_A=\coprod_{i=1}^{h}\U(L)g_i \SO(V)
\]
and put $\Gamma_i=g_i^{-1}\U(L)g_i\cap \SO(V)$.
This is a finite group, equal to 
the subgroup of $\SO(V)$ preserving $L_i=Lg_i$.
Denote by $W^{\Gamma_i}$ the space of vectors in $W$ 
fixed by $\Gamma_i$. Then we have  
\[
{\mathfrak M}_{\rho}(\cL)\cong \oplus_{i=1}^{h}W^{\Gamma_i},
\]
where $h$ is the class number of the genus $\mathcal{L}$ containing $L$.
Note that $[\O(V):\SO(V)]=2$ and 
$\O(L_i)$ is not generally $\SO(L_i)$.
In that case, $\O(L_i)$ acts on $W^{\Gamma_i}$
as $\{\pm 1\}$ and we have 
\[
W^{\Gamma_i}=W^{\O(L_i)}\oplus W^{\O(L_i),det}
\]
where $W^{\O(L_i),det}$ is the space on which 
any $g \in \O(L_i)$ with $\det(g)=-1$ acts as 
$-1$. 

In order to treat algebraic modular forms coming from the Spin group using only the group $\SO(V)$, 
one has to generalize the definition of algebraic modular forms to allow other characters.
For any $\{\pm 1\}$-valued character 
$\chi$ of $\SO(V)$, we write 
\[
\mathfrak{M}_{\rho}(\cL,\chi)=\oplus_{i=1}^{h}
W^{\Gamma_i,\chi},
\]
where we mean 
\[
W^{\Gamma_i,\chi}=\{w\in W; \rho(\gamma_i)w=\chi(\gamma_i)w 
\text{ for all }\gamma_i\in \Gamma_i\}.
\]
The adelic meaning of this definition will 
be explained later.
In particular, assume that $\rho$ is the spherical 
representation of degree $\nu$ of $\SO(m)$ corresponding to the dominant integral weight  $(\nu,0,\ldots,0)$ (i.e. the Young diagram 
in which the first row is of length $\nu$ and the other rows are empty.) Then $W$ is nothing but the space of homogeneous harmonic polynomials 
in $m$ variables of degree $\nu$. 
For applications to Hermitian modular forms, we will only consider the case $\mathrm{dim}(V) = 6$.
Then $-1_6\in \Gamma_i$ for every $i$, so a harmonic polynomial $P(x)=P(x_1,x_2,x_3,x_4,x_5,x_6)$ that occurs in a modular form of weight $\nu$ with respect to 
$\Gamma_i$ must satisfy $P(-x)=P(x)$ and in particular $P=0$ unless 
$\nu$ is even. Harmonic polynomials of odd degree $\nu$ satisfy
$P(-x)=-P(x)$, which can occur in modular forms in $\mathfrak{M}_{\nu}(\cL,\chi)$ for a character $\chi$ for which $\chi(-1_6)=-1$. This is often the case for spin characters, which we define later on.
\\

We denote by $C_e(Q)$ the even Clifford algebra of $Q$ and define the even Clifford group $\Gamma_e(Q)$ by 
\[
\Gamma_e(Q)=\{g\in C_e(Q)^{\times}:\; gVg^{-1}=V\}.
\]

The Clifford algebra $C(Q)$ 
has a natural antiautomorphism $J$ and the norm of $g\in \Gamma_e(Q)$ is defined as $N(g)= g \cdot J(g)$.
For $g\in \Gamma_e(Q)$, we have 
$N(g)\in \Q^{\times}$ and 
$N(\Gamma_e(Q))=\Q^{\times}_+$. It is well known that $\Gamma_e(Q)/\Q^{\times}\cong \SO(V)$ (see \cite{kneser}). \\

The spinor group is defined as 
\[
\Spin(Q)=\{g\in \Gamma_e(Q):N(g)=1\}.
\]
The spinor norm of the pullback of an element $g\in \SO(V)$ is determined up to $((\Q)^{\times})^2$ and this $(\Q^{\times})^2$-coset is also called the spinor norm $N(g)$ of $g$.
For a quadratic lattice $L\subset V$, denote by 
$d(L)$ the discriminant of $L$ (i.e. $\det((e_i,e_j)$ for a basis $\{e_1,\ldots,e_n\}$ 
of $L$), and denote by $\SO(L)$ the group of orientation-preserving automorphisms
of $L$ that respect the quadratic form $Q$.
By Chapter 10, 
Lemma 4.1 of \cite{cassels}, locally at any prime $q$, the spinor norms of
elements of the automorphism group of a
unimodular lattice are $q$-adic units modulo squares. 
So the (squarefree part of the) 
spinor norm of an element of $SO(L)$ for a global lattice $L$ 
is a product of prime divisors of $d(L)$. 

To extract the spinor norm one part from $\SO(L)$
for a lattice $L$ with discriminant $|\Delta_K|$, 
the spinor characters of $\SO(V)$ are defined as follows. Call the fundamental 
discriminant of a quadratic field over $\Q$ a prime discriminant if it is 
divisible only by one prime.
Then the discriminant $\Delta_K$ of $K$ splits as a product of distinct prime discriminants $d_i$, say
$$\Delta_K=d_1\cdots d_r.$$ For any product $d_0$ of the $d_i$ (i.e. any fundamental 
discriminant $d_0|\Delta_K$, but including $d_0=1$), we define $\nu_{d_0}$ as the character of $\Q^{\times}_+/(\Q^{\times})^2$ 
for which, for a prime $p$, we have 
\[
\nu_{d_0}(p)=\left\{\begin{array}{ll}
-1 & \text{ if } p|d_0 \\
1 & \text{ otherwise} 
\end{array}\right.
\]
For each $a\in \SO(V)$, define the spinor norm $N(a)$
of $a$ as usual and define 
$$\spin_{d_0}(a)=\nu_{d_0}(N(a)).$$
This character can be naturally extended to $\SO(V)_A$ by acting trivially on $\SO(V)_{\infty}$ and $\SO(V)_p$ for $p\nmid \Delta_K$.
So we put 
\[
U_0(L)=\{u=(u_p)\in U(L):\, \spin_{d_0}(u_p)=1 
\text{ for all } d_0|\Delta_K\}.
\]
Then $[U(L):U_0(L)]=2^t$, where $t$ is the number of prime
divisors of $\Delta_K$. 
We consider algebraic modular forms of weight $\nu$ with 
respect to $U_0(L)$ and denote the space of these forms by 
$\fM_{\nu}(\Spin(\cL_K))$.
For any $u\in U(L)$ and any $f \in \fM_{\nu}(\Spin(\cL_K))$, we define a map $T(u)f$ of $SO(V)_A$ to 
$W$ by 
\[
(T(u)f)(g)=f(u^{-1}g).
\]
It is easy to show that $T(u)f \in \fM_{\nu}(\Spin(\cL_K))$. Since $u^2\in U_0(L)$, we have $T(u)^2f=T(u^2)f=f(u^{-2}g)=f(g)$, so $T(u)^2=id$ as 
a linear transformation of $\fM(\Spin(\cL_K))$.
Since $U(L)/U_0(L)$ is an elementary abelian 2-group, 
$T(u)$ is simultaneously diagonalizable. 
So if we label the eigenspace with respect to $\spin_{d_0}$ by 
\[
\fM(\cL_K,spin_{d_0})=\{f\in \fM_{\nu}(\Spin(\cL_K));f(u^{-1}g)=spin_{d_0}(u)f\},
\]
then we have a common eigenspace decomposition
\[
\fM_{\nu}(\Spin(\cL_K))=\bigoplus_{\begin{subarray}{c}
d_0|\Delta_K \\ d_0;disc\end{subarray}}\fM(\cL_K,\spin_{d_0}).
\]
Given a double coset decomposition 
$\SO(V)_A=\coprod_{i=1}^{h}U(L)g_iSO(V)$,
we can define a mapping
\[
\fM_{\nu}(\cL_K,\spin_{d_0})\rightarrow 
\bigoplus_{i=1}^{d}W^{\Gamma_i,\spin_{d_0}}
\qquad 
\text{ by }
f\rightarrow (\rho(g_i)^{-1}f(g_i))_{1\leq i\leq h}.
\]
It is easy to see that this is an isomorphism. 
So we have 
\[
\fM_{\nu}(\Spin(\cL_K))=\oplus_{i=1}^{h}W^{\Gamma_i,\spin_{d_0}}
\]
This justifies the definition above.
Note that it can happen that 
$W^{\Gamma_i,\spin_{d_0}}=0$.

The space 
$\fM_{\nu}(\Spin(\cL_K))$
is what we will eventually want to compare with
the space of Hermitian modular forms of weight $\nu+4$.
A suitable subspace  of this should correspond 
with Hermitian modular forms of the maximal discrete extension of $\Gamma_K$ as defined in \cite{kriegraumwernz}.

Spinor characters can be calculated as 
follows. For any vector $x \in V \backslash \{0\}$, the
reflection $\tau_x\in O(V)$ with respect to $x$ is defined by 
\[
\tau_x(y)=y-\frac{B(x,y)}{Q(x)} x.
\]
Any element of $g\in \SO(V)$ can be written 
as a product of an even number of reflections, $g=\tau_{x_1}\cdots \tau_{x_{2r}}$ with
respect to vectors $x_i\in V$. 
An algorithm to calculate such a decomposition 
is given in \cite{cassels}.
Then $$N(g)=Q(x_1)\cdots Q(x_{2r}) \in \mathbb{Q}^{\times} / (\mathbb{Q}^{\times})^2$$
is the spinor norm of $g$. One can show in particular that the spinor norm of $-1_6$ is 
$-\Delta_K$ modulo $(\Q^{\times})^2$ (See \cite{kitaoka} p. 30, Exercise 1.) 
This implies for example that if $\Delta_K=-p$ for an odd prime $p$, 
then $\spin_p(-1_6)=-1$ and therefore $\fM_{\nu}(\cL,\spin_p)=0$ if $\nu$ is even.
In particular, if $\Delta_K$ is a prime discriminant, then we have 
\[
\fM_0(\Spin(\cL_K))=\fM_0(\cL_K)
\]
This is not true for general discriminants.

\subsection{Hecke operators on algebraic modular forms}

As before, let $L \subseteq V$ be an integral lattice of rank $m$ and let $W = W_{\nu}$ be the representation space of harmonic polynomials of homogeneous degree $\nu$.

Integral lattices $L$ and $M$ contained in a common ambient space are called \emph{Kneser $p^k$-neighbors} ($p$ prime, $k \ge 1$) if $L \cap M$ has index $p^k$ in both $L$ and $M$ and if $$L / (L \cap M) \cong M / (L \cap M) \cong (\mathbb{Z}/p \mathbb{Z})^k$$ are $p$-elementary. Equivalently, if there is a basis $e_1,...e_k, f_1,...,f_k, g_1,...,g_n$ of $L_p = L \otimes \mathbb{Z}_p$ with $$\langle e_i, f_j \rangle = \delta_{ij} \quad \text{and} \quad \langle e_i, e_j \rangle = \langle f_i, f_j \rangle = \langle e_i, g_j \rangle = \langle f_i, g_j \rangle = 0,$$ such that \begin{equation} \label{eq:matM} p^{-1} e_1,...,p^{-1} e_k, p f_1,...,p f_k, g_1,...,g_n \end{equation} is a basis of $M_p = M \otimes \mathbb{Z}_p$.

$p^k$-neighbors belong to the same genus and are clearly equivalent (even in the narrow sense) at all primes $\ell \ne p$.
The $p^k$-neighbor relation is interesting only when $k \le \frac{1}{2} \mathrm{rank}(L)$, and for $k = \frac{1}{2}\mathrm{rank}(L)$ only when $L_p$ is split. \\

For any prime $p$ and any $k \le \frac{1}{2} \mathrm{rank}(L),$ consider the double coset decomposition $$U\alpha_{p, k} U = \bigcup_{\gamma} \gamma U.$$ 
Then for any $g \in \mathrm{SO}(V)_A$, the lattices $L (\gamma g)$ are precisely the Kneser $p^k$-neighbors of the lattice $Lg$.

The Hecke operator $T_{p, k}$ is defined on $\mathfrak{M}_{\rho}(L)$ by $$T_{p, k} f(g) = \sum_{\gamma} \rho(\gamma) f(\gamma^{-1} g)$$ 
(See \cite{hashimoto} p. 230 (3)). By construction it is independent of the choice of $\gamma$, and for $u \in U$ and $a \in \mathrm{SO}(V)$ we have $T_{p, k} f(uga) = \rho(u) \cdot T_{p, k}(g)$.

In practice, we compute the action of $T_{p, k}$ at a prime $p \nmid \mathrm{det}(L)$ as follows. Suppose we fix representatives $L_i$ of the genus $\mathcal{L}$ such that $L = L_1$, together with fixed $p$-adic isometries $\beta_i : L_p \stackrel{\sim}{\rightarrow} (L_i)_p$. Elements $f \in \mathfrak{M}_{\nu}(V, \chi)$ are then represented by sequences of polynomials $(P_i)$ of homogeneous degree $\nu$, with each $P_i$ harmonic with respect to $L_i$. 

To apply $T_{p, k}$ to $f$, we have to compute the $p^k$-neighbors $M$ of $L_i$ (following \cite{GreenbergVoight}, which amounts to computing $k$-dimensional isotropic subspaces of $L$ over $\mathbb{F}_p$), together with a change-of-basis matrix $\gamma : L_i \rightarrow M$ corresponding to a choice of bases as in (\ref{eq:matM}). For each neighbor $M$ we find a global isometry $\alpha_j : M \stackrel{\sim}{\rightarrow} L_j$. Then $\beta_j^{-1}\alpha_j\gamma\beta_i$ belongs to $\mathrm{O}(L)$, and $P_i(\alpha_j \gamma X)$ is harmonic with respect to $L_j$, and $T_{p, k} f$ is represented by the harmonic polynomials $$\tilde P_j(X) = \sum_i \sum_{\substack{p^k\text{-neighbors} M \text{of} L_i \\ M \cong L_j}} \chi(\beta_j^{-1} \alpha_j \gamma \beta_i) P_i \Big( \alpha_j \gamma X \Big).$$

The computationally difficult parts of this process are testing lattices for isometry (and producing explicit isometries), which we carried out using the algorithm of Plesken and Souvignier \cite{PS97} provided by SageMath, and letting those isometries act on the defining polynomials, which becomes quite tedious when the degree is large.
The procedure to compute the $p^k$-neighbors is described in detail in Chapter 5 of Hein's thesis \cite{Hein}.

\subsection{L-functions for orthogonal modular forms}

The Euler factors in the standard $L$-function of an eigenform on $\mathrm{SO}(6)$ at good primes were computed by Murphy (\cite{Murphy}, Section 3.3) by means of the Satake transform. We have to modify Murphy's result slightly to get the algebraic normalization of the $L$-function. Let $\lambda_{p, k}$ be the eigenvalue of $f \in \mathfrak{M}_{\nu}(V, \chi)$ under $T_{p, k}$. Then $L^{\mathrm{std}}(f, s) = \prod_p L_p(f; p^{-s})$, where:

(i) If $L$ is split at $p$, then \begin{align*} L_p(f; X) &= 1 - \lambda_{p, 1} X + p^{\nu + 1} (p^{\nu} + p^{\nu+1} + p^{\nu + 2} + \lambda_{p, 2}) X^2 - p^{2\nu + 3} (2 \lambda_{p, 1} + \lambda_{p, 3}) X^3 \\ &\quad + p^{3\nu + 5} (p^{\nu} + p^{\nu+1} + p^{\nu+2} + \lambda_{p, 2}) X^4 - p^{4\nu + 8} X^5 + p^{6\nu + 12} X^6.\end{align*}

(ii) If $L$ is not split at $p$, then
\begin{align*}
    L_p(f; X) &= (1 - p^{2\nu + 4} X^2) \\ &\times (1 - \lambda_{p, 1} X + p^{\nu+1} (p^{\nu} - p^{\nu+1} + p^{\nu+2} + p^{\nu + 3} + \lambda_{p, 2})X^2 - p^{2\nu + 4} \lambda_{p, 1} X^3 + p^{4\nu + 8} X^4).
\end{align*}

\subsection{The theta map}

The \emph{theta map} is a function that takes algebraic modular forms on $\SO_n$ as input and yields classical modular forms (for subgroups of $\SL_2(\mathbb{Z})$) that are sums of theta functions.

An algebraic modular form $f$ of weight $\nu$ is represented by a sequence of polynomials $(P_i)$ of homogeneous degree $\nu$, each harmonic with respect to some (fixed) representative $L_i$ of the narrow equivalence classes in a given genus of rank $n$. The discriminant forms $\mathcal{D} = L_i'/L_i$ can be identified with one another under $\mathrm{SO}(V)$. It is therefore natural to consider the (vector-valued) theta functions, $$\theta_{P_i}(z) = \sum_{\lambda \in L_i'} P_i(\lambda) e^{\pi i \langle \lambda, \lambda \rangle z} e_{\lambda + L_i},$$ where $e_{\gamma}$ is the natural basis of the group ring $\mathbb{C}[\mathcal{D}]$. Since $P_i$ is harmonic, this is a modular form of weight $\nu + n/2$ with respect to the Weil representation of $\mathcal{D}$. Moreover, for any automorphism $\gamma \in \mathrm{O}(\mathcal{D})$ of the discriminant form, the function $$\gamma^* \theta_{P_i}(z) = \sum_{\lambda \in L_i'} P_i(\lambda) e^{\pi i \langle \lambda, \lambda \rangle z} e_{\gamma(\lambda + L_i)}$$ is again a modular form. The vector-valued setting allows us to define theta functions for polynomials of odd degree. \\

 The \emph{theta map} is then defined by $$\theta(f) := \frac{1}{\#\O(\mathcal{D})} \sum_i \frac{1}{\#\SO(L_i)} \sum_{\gamma \in \mathrm{O}(L_i'/L_i)} \gamma^* \theta_{P_i}.$$

For example, by the Siegel--Weil formula the image of the constant function $1$ under the theta map is the vector-valued Eisenstein series of weight $n/2$. See for example the discussion in Section 3.7 of \cite{Opitz}.

It is an important fact that the theta map is equivariant with respect to the Hecke operators on $\SO_n$ and $\SL_2$ at all primes not dividing the discriminant. In particular, the kernel $\ker(\theta)$ is a space of algebraic modular forms that is closed under those Hecke operators. 

For more precise statements about the theta map, we refer to Section 5 of \cite{AFILSV}.

\section{Conjectures}\label{conjecture}

 Our conjectures relate Hermitian modular forms for the full group $\Gamma_K$ with algebraic modular forms on a genus  $\mathcal{L}_K$ of rank six lattices associated to the field $K$. \\

Recall (by Nikulin, \cite{N1980}, Corollary 1.9.4) that the genus of an even integral lattice $L$ is uniquely determined by the signature $(r, s)$ and the discriminant form $(L'/L, q)$.

For any imaginary-quadratic field $K$, there is a natural discriminant form: view $\mathcal{O}_K$ as an even integral lattice with respect to the norm $N_{K/\mathbb{Q}}$. Let $\mathcal{L}_K$ be the genus of positive-definite even lattices $L$ of rank six with discriminant form $(L'/L, q) \cong (\mathcal{O}_K'/\mathcal{O}_K, -N_{K/\mathbb{Q}})$. 
Equivalently (by Nikulin, cf. \cite{N1980}, Theorem 1.12.2), $\mathcal{L}_K$ is the genus of even lattices $L$ that admit an isometry $$L \oplus H \oplus H 
 \cong E_8 \oplus \mathcal{O}_K(-1),$$ where $H$ is the standard hyperbolic plane and $E_8$ is the usual $E_8$-lattice, and where $\mathcal{O}_K(-1)$ is the lattice $\mathcal{O}_K$ with quadratic form $-N_{K/\mathbb{Q}}$.
At any prime $p$, we have 
$E_8 \cong H\oplus H \oplus H\oplus H$ by the 
classification theory of even unimodular lattices
\cite{kitaoka}. So over $\mathbb{Z}_p$, any such lattice $L$ satisfies
\[
L_p\oplus H \oplus H = H\oplus H \oplus (H\oplus H \oplus {\mathcal O}_K(-1)).
\]
By Witt's theorem (\cite{kitaoka} Theorem 5.3.5), we can cancel $H\oplus H$ to obtain the local isometry classes $$L_p = H \oplus H \oplus \mathcal{O}_K(-1)$$ of the genus $\mathcal{L}_K$. \\

Our first conjecture actually predicts the existence of certain distinguished algebraic modular forms on $\mathrm{SO}(6)$ which do not correspond to Hermitian modular forms. We formulate it here only for $D \in \{3, 4, 7, 8, 11\}$, since the precise statement for general discriminant is not yet clear:

\begin{conj}\label{conj:mainconj}
Let $f, g \in S_{\nu + 3}(\mathrm{SL}_2(\mathbb{Z}))$ be cuspidal eigenforms of degree one and the same weight $\nu+3$. (In particular, $\nu$ is odd.) \\
Suppose $D \in \{7, 8, 11\}$ and define an eigenform $h$ of weight 3 and level $D$ as follows:\\
(i) ($D = 7$) Let $h$ be the cusp form of level $7$ and weight $3$ with CM by $\mathbb{Q}(\sqrt{-7})$: $$h(\tau) := \eta^3(\tau) \eta^3(7\tau) = q - 3q^2 + 5q^4 - 7q^7 \pm ...$$
(ii) ($D = 8$) Let $h$ be the cusp form of level $8$ and weight $3$ with CM by $\mathbb{Q}(\sqrt{-8})$: \begin{align*} h(\tau) &:= \eta^2(\tau) \eta(2\tau) \eta(4\tau) \eta^2(8\tau) \\ &= q - 2q^2 - 2q^3 + 4q^4 + 4q^6 \pm ... \end{align*}
(iii) ($D = 11$) Let $h$ be the cusp form of level $11$ and weight $3$ with CM by $\mathbb{Q}(\sqrt{-11})$: $$h(\tau) = q - 5q^3 + 4q^4 - q^5 + 16q^9 \pm ...$$

Then there is an algebraic modular form $\mathrm{Y}_{f, g, h}$ on $\mathrm{SO}(\cL_D)$, of weight $\nu$ and spinor character $\mathrm{spin}_D$ whose standard $L$-function is $$L(\mathrm{Y}_{f, g, h}; s) = L(f \otimes g; s) L(h; s - \nu - 1).$$
\end{conj}
The notation $\mathrm{Y}_{f, g, h}$ is meant to suggest that $\mathrm{Y}$ is an analogue of the Yoshida lift for the modular forms $h$ and $f \otimes g$.

When $f \ne g$, we experimentally find two distinct Yoshida lifts of the triple $(f, g, h)$ with the same $L$-function. They can be distinguished by passing from $\SO(6)$ to $\O(6)$, however: one lift always has the determinant character, and the other does not.

More generally, if $D = d_1 d_2$ where $d_1 > 0$ is a fundamental discriminant, then our  computations suggest that the Yoshida lift takes Hilbert modular eigenforms $F$ of parallel weight $(\nu+3, \nu+3)$ and level one attached to the real-quadratic field $\mathbb{Q}(\sqrt{d_1})$ and eigenforms $h$ on $\Gamma_1(d_2)$ with CM by $\mathbb{Q}(\sqrt{-d_2})$, and it produces an algebraic modular form of weight $\nu$ whose standard $L$-function is $$L(\mathrm{Y}_{F, h}; s) = L^{\mathrm{Asai}}(F; s) L(h; s - \nu - 1).$$ Here $L^{\mathrm{Asai}}(F; s)$ is Asai's $L$-function \cite{Asai1977}. From this point of view, $f \otimes g$ should be thought of as a ``split" Hilbert modular form with $d_1 = 1$, i.e. as the modular form $f(z_1) g(z_2)$ for $\mathrm{SL}_2(\mathbb{Z}) \times \mathrm{SL}_2(\mathbb{Z})$ and the product $L$-function $L(f \otimes g; s)$ should be viewed as its Asai $L$-function. We will refrain from making a general conjecture as to which Yoshida lifts $\mathrm{Y}_{F, h}$ occur in $S_k(\mathrm{SO}(6))$. \\

\begin{rem} When $d_2 = p$ with a prime $p \equiv 3 \, (4)$, then a result of Bruinier--Bundschuh \cite{BB2003} implies that weight three eigenforms $h$ on $\Gamma_1(d_2)$ with CM by $\mathbb{Q}(\sqrt{-d_2})$ are in bijection with eigenforms of weight three for the Weil representation of $\mathrm{SL}_2(\mathbb{Z})$ attached to any of the rank six lattices $L$ in the genus $\mathcal{L}_K$, hence (via the theta decomposition) in bijection with Jacobi eigenforms of weight six for the lattice $L$. This ``lifting" from CM eigenforms $h$ to Jacobi eigenforms also works for non-prime $d_2$.

It is tempting to guess that such Jacobi forms are actually the underlying objects being lifted, since they are clearly associated with the genus $\mathcal{L}_K$.
\end{rem}

\begin{ex}
    Let $K = \mathbb{Q}(\sqrt{-7})$, with corresponding signature $(6, 0)$ lattice $A_6$. Then $\fM_9(\cM, \spin_7)$ contains a unique rational eigenform and its Euler factors at small primes $p \ne 7$ have the following form: 
    \begin{small} \begin{align*} L_2(X) &= (1 - 2^{11} X)^2 (1 + 3 \cdot 2^{10} X + 2^{22} X^2) (1 + 3520 X + 2^{22} X^2) \\ L_3(X) &= (1 - 3^{11} X)^2 (1 - 3^{22} X^2) (1 + 290790 X + 3^{22} X^2) \\ L_5(X) &= (1 - 5^{11} X)^2 (1 - 5^{22} X^2) (1 + 74327350X + 5^{22} X^2) \\ L_{11}(X) &= (1 - 11^{11} X)^2 (1 + 6 \cdot 11^{10} X + 11^{22} X^2) (1 + 284813350678 X + 11^{22} X^2).\end{align*}\end{small}
These Euler factors exactly match the conjectural lift $\mathrm{Y}_{f, g, h}$ where $f = g = \Delta(\tau)$ and $h(\tau) = \eta^3(\tau) \eta^3(7\tau).$
\end{ex}

\begin{ex}
    Again, let $K = \mathbb{Q}(\sqrt{-7})$. The genus $\mathcal{L}_K$ is represented by the $A_6$ root lattice, whose orthogonal group is the Weyl group $W(A_6)$. The fundamental invariants $p_2,...p_7$ of $W(A_6)$ have degrees $2,3,4,5,6,7$, and their Jacobian determinant $$\Psi = \mathrm{det}\Big( \nabla p_2,..., \nabla p_7 \Big)$$ is a harmonic polynomial that transforms under $W(A_6)$ with the determinant character. Hence it defines an algebraic modular form for $\mathrm{SO}(6)$ of weight $21$ and spinor character $\mathrm{spin}_7$, and also the determinant character if it is viewed as a modular form on $\O(6)$. Since the space of such forms is one-dimensional, it is automatically a rational eigenform. Its Euler factors at small primes $p \ne 7$ factor in the form
    \begin{tiny}
    \begin{align*} L_2(X) &= (1 + 3 \cdot 2^{22} X + 2^{46} X^2) \cdot (1 + 2^{10} \cdot 19989 X + 2^{29} \cdot 395729 X^2 + 2^{56} \cdot 19989 X^3 + 2^{92} X^4) \\ L_3(X) &= (1 - 3^{46} X^2) \cdot (1 + 3^7 \cdot 8696912 X - 3^{27} \cdot 432091286 X^2 + 3^{53} \cdot 8696912 X^3 + 3^{92} X^4)  \\ L_5(X) &= (1 - 5^{46} X^2) \cdot (1 + 5^4 \cdot 5378001278076 X - 5^{25} \cdot 471210985152034 X^2 + 5^{50} \cdot 5378001278076 X^3 + 5^{92} X^4) \\ L_{11}(X) &= (1 + 6 \cdot 11^{22} X + 11^{46} X^2) (1 - 11^2 \cdot 1260169201464060096144 X - 11^{25} \cdot 11253816697881924170134 X^2 \\ &\quad \quad \quad \quad \quad \quad \quad \quad \quad \quad \quad \quad \quad \quad - 11^{48} \cdot 1260169201464060096144 X^3 + 11^{92} X^4). \end{align*}
    \end{tiny}
    These Euler factors exactly match the conjectured Yoshida lift $\mathrm{Y}_{f,g,h}$, where $h(\tau) = \eta^3(\tau) \eta^3(7\tau)$ and where $f$ and $g$ are the distinct eigenforms of weight 24 for $\mathrm{SL}_2(\mathbb{Z})$. \\
    As mentioned above, there is a different algebraic modular eigenform for $\O(6)$ of weight 21 and spin character $\spin_7$, \emph{without} the determinant character, which has exactly the same Euler factors at the primes listed above. This can be calculated (with some difficulty) using \cite{WeilRep}.
\end{ex}

Our main conjecture is that, apart from Miyawaki lifts and Yoshida lifts, eigen cusp forms of weight $k$ on $\mathrm{SU}(2, 2)$ correspond exactly to non-constant eigenforms of weight $k-4$ on $\mathrm{SO}(6)$.

\begin{conj}[Main conjecture] \label{conj:main}
There is a one-to-one correspondence between cuspidal Hermitian eigenforms $F \in S_k(\Gamma_K)$ that are not Miyawaki lifts and nonconstant algebraic eigenforms $G \in \fM_{\nu}(\mathrm{Spin}(\mathcal{L}_D))$, $\nu = k - 4$, that are not Yoshida lifts. It has the following properties: \\ 
(i) The degree six zeta function $L(F; s)$ equals the standard $L$-function $L(G; s)$ (possibly up to Euler factors at primes $p | \Delta_K$). \\
(ii) The spinor character of $G$ is $$\chi_G = \prod_{F | W_p = -F} \mathrm{spin}_p,$$ i.e. $\mathrm{spin}_p$ occurs in $\chi_G$ if and only if $F$ has eigenvalue $-1$ under the Atkin--Lehner involution $W_p$. \\
(iii) $F$ belongs to the Sugano Maass space if and only if $G$ does \emph{not} belong to the kernel of the theta map (cf. Section 2.4).
\end{conj}

Since the theta map is Hecke-equivariant, part (iii) of Conjecture \ref{conj:main} would follow from the Eichler basis problem for modular forms for the Weil representations of the genera $\mathcal{L}_K$. More precisely, the Hermitian modular form $F$ should simply be the Maass lift of the vector-valued modular form $\theta(G)$. The basis problem was recently shown by M\"uller \cite{Mueller} to have a positive solution for lattices of rank at least seven, and similar methods might show that it holds for lattices of rank six.

\section{The mass formula}

In the rest of the paper we will give some evidence that makes Conjecture \ref{conj:main} plausible.
The main observation is that the asymptotic growth of the dimensions of spaces of Hermitian and algebraic modular forms are equal.

For Hermitian modular forms, the asymptotic dimension formula is easily given in terms of a Bernoulli number \eqref{eq:asymptotic}. To compare this with $\SO(6)$ we have to calculate the mass of the genus $\mathcal{L}_K$.

Note however that we consider algebraic modular forms
with respect to $\SO(6)$, but the notion of genus is usually taken with respect to $\O(6)$-equivalence classes;  we will briefly discuss the difference.
Let $V$ be a finite-dimensional vector space over $\Q$, equipped with a positive definite quadratic form $Q$; let $\O(V)$ be the orthogonal group of $Q$ and $\SO(V) = \O(V) \cap \SL(V)$. The adelizations of these groups are denoted by $\O(V)_A$ and $\SO(V)_A$, and their localizations at a place $v$ by $\O(V)_v$ and $\SO(V)_v$.
For any lattice $L\subset V$ and a finite place $v$ 
of $\Q$, put $L_v=L\otimes_{\Z}\Z_v$, where $\Z_v$ is 
the ring of $v$-adic integers. 
For a fixed lattice $L$ of $V$ and an element $g=(g_v)\in 
\O(V)_A$, we define a lattice $Lg$ of $V$ by 
\[
Lg=\cap_{v<\infty}(L_vg_v\cap V).
\]
Then ${\mathcal L}=\{Lg;g\in \O(V)_A\}$
is the genus containing $L$. We have
$$\mathcal{L} = \{Lg;g\in \SO(V)_A\}$$ because every lattice over $\Z_v$ has an automorphism of determinant $-1$. This is clear if $v\neq 2$ since $L$ is diagonalisable. If $v=2$, then the lattice can be decomposed 
into a direct sum of rank-one and rank-two lattices, where the Gram matrix of the binary lattice is either $2^n\begin{pmatrix} 0 & 1 \\ 1 & 0 \end{pmatrix}$ or 
$2^n \begin{pmatrix} 2 & 1 \\ 1 & 2 \end{pmatrix}$, and in both cases the matrix $\begin{pmatrix} 0 & 1 \\ 1 & 0 \end{pmatrix}$
is such an automorphism. 
Hence if $M=Lg$ for $g\in \O(V)_A$, 
then we also have $M=Lg_0$ for some $g_0\in \SO(V)_A$.

The set $\mathcal{L}$ decomposes into $O(V)$-equivalence classes, and the $\O(V)$-equivalence classes might decompose into finer $\SO(V)$-equivalence classes. In both cases, there are only finitely many classes and their number is called the class number in the wide sense and in the narrow sense, respectively.
Denote by $M_1, \ldots, M_m$ the representatives of $\O(V)$ classes in $\cL$ and by $L_1$, \ldots, $L_h$ 
the representatives of $\SO(V)$ classes in $\cL$.
Clearly we have $m\leq h$. 
For any lattice $M$, we put 
\[
\O(M)=\{g\in \O(V);Mg=M\},
\quad 
\SO(M)=\{g\in \SO(V);Mg=M\}.
\]
The mass of the genus is defined by
\[
M(\cL,\O(V))=\sum_{i=1}^{m}\frac{1}{\#\O(M_i)},
\qquad 
M(\cL,\SO(V))=\sum_{i=1}^{h}\frac{1}{\#\SO(L_i)}.
\]

\begin{lem}
We have 
\[
M(\mathcal{L},\SO(V))=2 \cdot M(\mathcal{L},\O(V)).
\]
\end{lem}
\begin{proof}
If $\O(L_i)$ contains an element of determinant 
$-1$, then no $L_j$, $j\neq i$ can be $\O(V)$-equivalent to $L_i$: if it were, then $L_j$ would also be 
$\SO(V)$-equivalent to $L_i$, a contradiction.
In this case, we have $[\O(L_i):\SO(L_i)]=2$.
So 
\[
\frac{2}{\#\O(L_i)}=\frac{1}{\#\SO(L_i)}.
\]
If $\O(L_i)$ does not have an element with determinant $-1$, then the lattice $L_i g$ for any element 
$g\in \O(V)$ with $\det(g)=-1$ is never $\SO(V)$-equivalent to $L_i$ itself, so there is $j\neq i$ such that $L_j$ and $L_i$ are $\O(V)$-equivalent.
In this case, we have $\#\O(L_i)=\#\O(L_j)=\#\SO(L_i)=
\#\SO(L_j)$, so we have proved
\[
\frac{2}{\#\O(L_i)}=\frac{1}{\#\SO(L_i)}+\frac{1}{\#\SO(L_j)}. \qedhere
\]
\end{proof}

Write $K=\Q(\sqrt{-D})$ where $-D=\Delta_K$ is the discriminant of the imaginary quadratic field $K$. 
So $D=m$ with $m\equiv 3 \bmod 4$, or $D=4m$ with 
$2||m$ or $m\equiv 1 \bmod 4$. 
For any prime $p$, we define a lattice $M_p$ 
over $\Z_p$ by $M_p=H\oplus H\oplus O_K(-1)$.
Then we have $\det(M_p)=D$. 

One can show that the product of the Hasse invariants is $\prod_p inv_v(M_p)=1$, so by Proposition 2.1 of 
\cite{ibusaito} there exists a global lattice
$L\subset V$ such that $L_p=M_p$ for every prime.
We consider the genus $\cL_K$ of (positive definite) lattices of 
rank $6$ and determinant $D$  including the lattice $L$. 
Now we calculate the mass 
$M(\cL_K,O(V))$ of $\cL_K$ using the Minkowski--Siegel formula.
We quote the result from Kitaoka, \cite[Theorem 6.8.1]{kitaoka} Let $N$ be a lattice of 
rank $m\geq 2$ and $\cN$ the genus 
containing $N$. Let $dN$ be the discriminant of 
$N$.

\begin{thrm}[Minkowski--Siegel]
\[
M(\cN,O(V))=2\frac{dN^{(m+1)/2}}{\pi^{m(m+1)/4}}
\prod_{i=1}^{m}\Gamma\left(\frac{i}{2}\right)
\prod_{p}\alpha_p(N_p,N_p)^{-1}.
\]
\end{thrm}
$\alpha_p(N_p,N_p)$ is the local density, 
defined by 
\[
\alpha_p(N_p,N_p)=2^{-1}\lim_{r\rightarrow \infty}
p^{m(m+1)/2-m^2}\#A'_{p^r}(N_p,N_p).
\]
Here we denote by $S$ the gram matrix of $N_p$ and we put
\[
A'_{p^r}(N_p,N_p)=\{X\in M_m(\Z_p)/p^rM_m(\Z_p);
^{t}XSX\equiv S \bmod p^r\}.
\]

The local densities $\alpha_p(M_p,M_p)$ are completely 
described by Theorem 5.6.3 in \cite{kitaoka}.
However, the formula is very complicated and it involves a lot of notation, which it seems better not to review here. 
We will freely use the notation of
Theorem 5.6.3 of \cite{kitaoka} below and only indicate what the values of the symbols are in our cases.

First note that $m=6$, $dM=D$, 
$\pi^{-6(6+1)/2}\prod_{i=1}^{6}\Gamma(i/2)=\frac{3\pi^9}{4}$.
So we have 
\[
M(\cL_K,O(V))=\frac{3D^{7/2}}{2\pi^9}
\prod_{p}\alpha_p(M_p,M_p)^{-1}.
\]
\begin{prop}\label{mass}
Denote by $t$ the number of prime divisors of 
$\Delta_K$. Then 
\[
M(\cL_K,O(V))=\frac{B_{3,\chi_K}}{2^{t-1}\cdot 2^9 \cdot 3^2\cdot 5}
\]
\end{prop}

\begin{proof}
For the sake of simplicity, we write 
$\alpha_p(M_p,M_p)=\alpha_p(M_p)$.

First assume that $D\equiv 3 \bmod 4$. If $p\neq 2$ and $p\nmid D$, then $$n_i=0, 
, (i\neq 0), 
\quad s=1, \quad w=0, \quad N_0=5(1)\oplus (D),$$ and $N_0$ 
is hyperbolic (i.e. isomorphic to $H\oplus H \oplus H$) if and only if $(-1/p)=(D/p)$,
where $(*/*)$ is the quadratic residue symbol. 
In other words, $\chi(N_0)=(-D/p)$. 
We have $$E=(1+(-D/p)p^{-3})^{-1}, \quad 
P=P(3)=(1-p^{-2})(1-p^{-4})(1-p^{-6}).$$
Considering the difference between $\beta_p$ and $\alpha_p$ in page 
98 of \cite{kitaoka}, we have 
\[
\alpha_p(M_p)=(1-p^{-2})(1-p^{-4})(1-(-D/p)p^{-3}).
\]
When $p\neq 2$ and $p|d$, then 
$$N_0=4(1)\oplus (-2), \quad N_1=(-D/2), \quad s=2, \quad w=1,$$  and $$E=1, \quad P=P(2)=(1-p^{-2})(1-p^{-4}),$$
and therefore
\[
\alpha_p(M_p)=2p(1-p^{-2})(1-p^{-4}).
\]
When $p=2$, then $$n_0=6, \quad q_j=0, \quad w=0,$$ and  
$$P=P(3)=(1-2^{-2})(1-2^{-4})(1-2^{-6}), \quad E_0=(1+(-D/2)2^{-3})/2,$$ where $(-D/2)$ is $+1$ if $-D \equiv 1 \bmod 8$ and $-1$ if 
$-D\equiv 5 \bmod 8$, and $E_j=1$ if $j\neq 0$, and therefore  
$$E=E_0^{-1}=2/(1+(-D/2)2^{-3}).$$ So we have 
\[
\alpha_p(M_p)=2^6(1-2^{-2})(1-2^{-4})(1-(-D/2))^{-3}.
\]
Altogether, denoting by $t$ the number of prime divisors of $\Delta_K$ and $\chi_{K}(n)=\left(\frac{\Delta_K}{n}\right)$, 
the mass of the genus is given by 
\[
M(\cL_K,O(V))=\frac{3|\Delta_K|^{5/2}}{2^{7+t}}
\zeta(2)\zeta(4)L(3,\chi_K).
\]
Since we have $\zeta(2)=\pi^2/6$, $\zeta(4)=\pi^4/90$ and 
\[
\frac{|\Delta_K|^{5/2}L(3,\chi_K)}{\pi^3}=\frac{2}{3}B_{3,\chi_K},
\]
this simplifies to
\[
M(\cL_K,O(V))=\frac{B_{3,\chi_K}}{2^{t-1}\cdot 2^9\cdot 3^3\cdot 5}.
\]
Next we consider the case $D=4m$, 
$m\equiv 1 \bmod 4$.
At $p\neq 2$, the local isometry class is
\[
M_p=H\oplus H \oplus 
\begin{pmatrix}-2 & 0 \\ 0 & -2m \end{pmatrix}.
\]
If in addition $p\nmid m$, then $M_p$ is 
unimodular and equivalent to 
$5(1)\oplus (4D)$. 
We have $$w=0, \quad s = 1,$$ $$P=(1-p^{-2})(1-p^{-4})(1-p^{-6}), \quad E=(1+\chi_K(p)p^{-3})^{-1}.$$ Hence
\[
\alpha_p(M_p)=(1-p^{-2})(1-p^{-4})(1-\chi_K(p)p^{-3}).
\]
For $p|m$ and $p\neq 2$, the Jordan decomposition is $M_p=N_0\oplus p \cdot N_1$ where 
by $$N_0=(1)\oplus (1) \oplus (1) \oplus (1) 
\oplus (-2)$$ and $N_1=(-m)$. 
We have $$s=2, \quad w=1, \quad P=(1-p^{-2})(1-p^{-4}),$$ and
$E=1$ (since $\mathrm{rank}(N_j)$ is odd and 
$\chi(N_j)=1$), and therefore
\[
\alpha_p(M_p)=2p(1-p^{-2})(1-p^{-4}).
\]
If $p=2$, then $M_p=N_0\oplus 2 N_1$ where 
$N_0=H\oplus H$ and 
$N_1=\begin{pmatrix} -1 & 0 \\ 0 & -m \end{pmatrix}$.
We have $$q_1=n_1=2, \quad w=3, \quad P=(1-2^{-2})(1-2^{-4}),$$ $$E_{-2}=1, \quad E_{-1}=1, \quad E_0=1/2,$$
$$E_1=1/2, \quad E_2=1/2, \quad E_3=1$$ and 
$$E=\prod_{j}E_j^{-1}=2^3,$$ and therefore
\[
\alpha_2(M_2)=2^9(1-2^{-2})(1-2^{-4}).
\]
Denoting by $t$ the number of prime divisors of $\Delta_K$, we obtain the mass
\[
M(\cL_K,O(V))=\frac{|\Delta_K|^{5/2}L(3,\chi_K)}
{\pi^3}\times \frac{1}{2^{t-1}\cdot 2^{10}\cdot 3^2\cdot 5}
=
\frac{B_{3,\chi_K}}{2^{t-1}\cdot 2^9\cdot 3^3\cdot 5}.
\]
Finally consider the case $\Delta_K=-4m$ where $2||m$. Write $m=2 m_0$.
If $p\neq 2$ and $p\nmid m_0$, then $M_p$ is unimodular and we have 
$$w=0, \quad s=1, \quad P=(1-p^{-2})(1-p^{-4})(1-p^{-6}),$$ and $M_p=N_0$ is hyperbolic if and only if the determinant is $-1 \bmod 8$, so 
$$\chi(N_0)=\chi_K(p), \quad E=(1+\chi_K(p)p^{-3})^{-1}.$$
So we have 
\[
\alpha_p(M_p)=(1-p^{-2})(1-p^{-4})(1-\chi_K(p)p^{-3}).
\]
If $p\neq 2$ and $p|m_0$, then we have $M_p=N_0\oplus p N_1$,
where $N_0=H\oplus H \oplus (-2)$ and
$N_1=(-4m_0/p)$. We have
$$n_1=5, \quad n_2=1, \quad s=2, \quad w=1,$$ 
$$P=(1-p^{-2})(1-p^{-4}), \quad \chi(N_0)=\chi(N_1)=0, \quad E=1.$$ So we have 
\[
\alpha_p(M_p)=2p(1-p^{-2})(1-p^{-4}).
\]
At $p=2$, we have $M_p=N_0\oplus 2N_1\oplus 4N_2$,
where $$N_0=H\oplus H, \quad N_1=(-1), \quad N_2=(-m_0).$$ Hence
$$s=3, \quad n_0=4, \quad n_1=n_2=1,$$ $$w=4, \quad q_0=1, \quad q_1=1, \quad q_2=1, \quad q=3,$$ 
$$P=(1-2^{-2})(1-2^{-4}),$$ $$E_{-2}=E_{-1}=1, \quad E_0=E_1=E_2=E_3=1/2,$$ $$E_4=1, \quad E=2^4.$$ Hence we have
\[
\alpha_2(M_2)=2^{10}(1-2^{-2}(1-2^{-4}).
\]
Altogether,
\[
\prod_{p}\alpha_p(M_p)^{-1}=\frac{\zeta(2)\zeta(4)L(3,\chi_K)}{2^{t-1}\cdot 2^{10}\cdot m_0}
\]
where $t$ is as before.
The mass of the genus is therefore \[
M(\cL_K,O(V))=
\frac{(8m_0)^{5/2}L(3,\chi_K)}{\pi^3}\times
\frac{1}{2^{t-1}\cdot 2^{10}\cdot 3^2\cdot 5}
=
\frac{B_{3,\chi_K}}
{2^{t-1}\cdot 2^9\cdot 3^3 \cdot 5}. \qedhere
\]
\end{proof}

For the spherical representation $(\nu,0,0)$ of $\SO(6)$ with 
even $\nu$, the character of $\pm 1_6$ is 
given by 
\[
\frac{(\nu+1)(\nu+2)^2(\nu+3)}{12},
\]
so multiplying double this to the mass of $\mathcal{L}_K$ with respect to $\SO(V)$, which is twice the mass 
obtained in Proposition \ref{mass}, 
we obtain the asymptotic dimension 
\[
\dim \fM_{\nu}(\cL_K)\sim \nu^4\cdot 
\frac{B_{3,\chi_K}}{2^{t-1}2^{9}\cdot 3^4\cdot 5}=
\nu^4\cdot \frac{B_{3,\chi_K}}{2^{t-1} \cdot 207360}, \quad \nu \rightarrow \infty, \; \nu \; \text{even}.
\]
Now we consider spinor characters.
Since the spinor norm of $-1_6$ is the same as 
$D=|\Delta_K|$ (by \cite{kitaoka} p.30, Exercise 1), 
we have $\spin_{d_0}(-1_6)=-1$ if and only if 
$d_0$ is divisible by an odd number of prime factors. 
So for example, for an odd prime $p|D$, 
the main term of 
$\fM_{\nu}(\cL_K)+\fM_{\nu}(\cL_K,\spin_p)$ 
is the same as that of $\fM_{\nu}(\cL_K)$:
\[
\dim \fM_{\nu}(\cL_K)+\dim \fM_{\nu}(\cL_K,\spin_p)\sim \nu^4\cdot 
\frac{B_{3,\chi_K}}{2^{t-1}\cdot 2^9\cdot 3^4\cdot 5}.
\]
Here the mass is doubled, but the kernel of $\spin_p$ does not contain $-1_6$ and so the main term only involves the contribution of $1_6$ and not $\pm 1_6$.
When $\Delta_K=-4$, then $\spin_4(-1_6)=1$ and we have 
\[
\dim \fM_{\nu}(\cL_K)+\dim \fM_{\nu}(\cL_K, \spin_p)\sim \nu^4\cdot 
\frac{B_{3,\chi_K}}{\cdot 2^8\cdot 3^4\cdot 5}.
\]

In general, when $\Delta_K\neq -4$, we have 
\[
\dim \fM_{\nu}(\Spin\, \cL_K) = \sum_{d_0|\Delta_D}\dim \fM_{\nu}(\cL_K,\spin_{d_0})
\sim \nu^4\cdot \frac{B_{3,\chi_K}}{2^9\cdot 3^4\cdot 5},
\]
where $d_0$ runs over fundamental discriminants 
or 1.

\section{Comparison with dimensions of Hermitian modular forms}
\subsection{Comparison of main terms}
In this subsection, we will see that the 
main terms of the trace formula for algebraic modular forms on $SO(6)$ with arbitrary spin characters and for Hermitian modular forms coincide. 
We also compare algebraic modular forms on $SO(6)$ without character 
and Hermitian modular forms for the maximal discrete extension $\Gamma_K^*$ of $\Gamma_K$.
These computations support our conjecture because the
dimensions of the spaces of lifts are of asymptotic order at most $k^3$, while the orders of the main terms are of order $k^4$.
We will also show that the dimensions are exactly equal in all cases where the actual dimensions are known. 

We abbreviate
\[
\fM_{\nu}(\Spin(\cL_K))=
\sum_{d|\Delta_K}\dim \fM_{\nu}(\cL_K, \spin_d),
\]
where $d$ runs over squarefree positive integers for which $d|\Delta_K$.

\begin{prop}\label{prop:dim_Mnu}
The main terms of $\dim \fM_{\nu}(\Spin(\cL_K))$ and 
$\dim M_{\nu+4}(\Gamma_K)$ coincide. They are  
equal to 

\[
\frac{B_{3,\chi_K}}{2^9\cdot 3^4\cdot 5}
(\nu+1)(\nu+2)^2(\nu+3)
\]
if $\Delta_K\neq -4$ and
\[
\frac{(\nu+1)(\nu+2)^2(\nu+3)}{69120}
\]
if $\Delta_K=-4$.

\end{prop}

For $\Delta_K = -4$, there are no modular forms of odd weight and the formula for even weight is twice what is predicted by the formula for general $\Delta_K$.

Let $\Gamma^*_K$ be the largest discrete subgroup of $\SU(2,2;\mathbb{C})$ that contains $\Gamma_K$. As described in Section \ref{subs:AL}, we have
\[
\Gamma_K^* = \bigcup_{\substack{d | \Delta_K \\ d \, \text{squarefree}}} \Gamma_K W_d
\]
where $W_d$ are certain Atkin--Lehner involutions, and $[\Gamma_K^{*}:\Gamma_K]=2^{t}$ where $t$ is the number of prime discriminants dividing $\Delta_K$.

\begin{prop}
The main terms of 
$\dim \fM_{\nu}(\cL_K)$ and $\dim M_{\nu+4}(\Gamma_K^*)$ coincide. They are given by 
\[
\frac{B_{3,\chi_K}}{2^{t-1}\cdot 2^9\cdot 3^4\cdot 5}
(\nu+1)(\nu+2)^2(\nu+3),
\]
if $\Delta_K \ne 4$ and \[
\frac{(\nu+1)(\nu+2)^2(\nu+3)}{69120}
\]
if $\Delta_K=-4$.
\end{prop}

\begin{proof}
We have $i 1_4\in \Gamma_K W_m$
by Remark 2 of \cite{kriegraumwernz}. 
If $K\neq \Q(\sqrt{-1})$, then $i 1_4\not\in \Gamma_K$ and we have 
\[
\mathrm{vol}(\mathbf{H}_2/\Gamma_K)=2^{t-1} \cdot \mathrm{vol}(\mathbf{H}_2/\Gamma_K^*).
\]
If $K=\Q(\sqrt{-1})$, then $i 1_4\in \Gamma_K$ and 
$[\Gamma_K^*:\Gamma_K]=2$.
The claim now follows from Proposition \ref{prop:dim_Mnu}.
\end{proof}

When $-\Delta_K$ is an odd prime, then 
$\Gamma_K^*=\Gamma_K \cup (i 1_4)\Gamma_K$, so in this case
\[
M_k(\Gamma_K^*)=\left\{
\begin{array}{cc}
M_k(\Gamma_K)     & \text{if $k$ is even.}  \\
    0 & \text{ if $k$ is odd. } 
\end{array}\right.
\]
Therefore $$\sum_{k=0}^{\infty} \mathrm{dim}\, M_k(\Gamma_K^*) t^k = \sum_{k=0}^{\infty} \mathrm{dim}\, M_{2k}(\Gamma_K) t^{2k}.$$

When $K=\Q(\sqrt{-1})$, then we have 
\[
\Gamma_K^*=\Gamma_K\cup 
\left(\frac{1+i}{2}\right)\begin{pmatrix} U & 0 \\ 0 & U \end{pmatrix}\Gamma_K \qquad \text{for }\quad U=\begin{pmatrix} 0 & 1 \\ 1 & 0 \end{pmatrix},
\]
and this simplifies to
$$\U(2,2;\cO_K):=\U(2,2;\C)\cap M_4(\cO_K).$$ In this case, 
\begin{align*}
\sum_{k=0}^{\infty}\dim M_k(\Gamma_K^*)t^k
& =
\sum_{k=0}^{\infty}\dim M_k(\U(2,2;\cO_k))t^k
\\ & = 
\frac{1+t^{44}}{(1-t^4)(1-t^6)(1-t^8)(1-t^{10})(1-t^{12})}.
\end{align*}
The last equality is due to Aoki \cite{A}.

\subsection{Comparison of dimensions for small discriminants}
Now we compare the dimensions of Hermitian modular forms and algebraic modular forms when the discriminant is $-3$, $-4$, $-7$, $-8$, or $-11$. These are the only cases where the dimensions of Hermitian modular forms are known in all weights.

For each even $k\geq 6$, we have one Siegel Eisenstein series, so the generating series of dimensions of Eisenstein series of weight $k\geq 6$ in $M_k(\Gamma_K)$ is given by 
\[
\Eis=\frac{t^6}{1-t^2}.
\]
We also have a non-cuspidal Hermitian modular form of weight $4$, also an Eisenstein series (although the naive definition of the Siegel Eisenstein series does not converge absolutely), but it corresponds to the constant algebraic modular form of weight zero.
The number of Klingen-type Eisenstein series of weight $k$ is equal to $\dim S_k(\SL_2(\Z))$, so their
generating series is 
\[
\KE=\sum_{k=0}^{\infty} \dim S_k(\SL_2(\Z)) t^k = \frac{t^{12}}{(1-t^4)(1-t^6)}.
\]
The dimensions of the spaces of Miyawaki lifts from $S_{k}(\SL_2(\Z))\times S_{k-2}(\SL_2(\Z))$
is
\[
\mathrm{MW} =\sum_{k=0}^{\infty}
\bigl(\dim S_k(\SL_2(\Z))\times \dim S_{k-2}(\SL_2(\Z))\bigr)t^k=
\frac{t^{18}}{(1-t^2)(1-t^6)(1-t^{12})}
\]
In particular, the contributions $\mathrm{Eis}$, $\mathrm{K}$, $\mathrm{MW}$ are independent of $\Delta_K$.

There is also the conjectural Yoshida-type lift that produces algebraic modular forms. These should never occur for discriminants $-3$ or $-4$. For discriminants $\Delta_K = -7$, $-8$ and $-11$ the only Yoshida-type lifts of weight $\nu$ involve eigenforms for $\mathrm{SL}_2(\mathbb{Z}) \times \mathrm{SL}_2(\mathbb{Z})$ of weight $\nu+3$ and the unique CM eigenform of level $\Gamma_0(|\Delta_K|)$ and weight three with quadratic character. We therefore put
\[
\mathrm{Y}=\sum_{\nu=0}^{\infty}
\bigl(\dim S_{\nu+3}(\SL_2(\Z))\times \dim S_{\nu+3}(\SL_2(\Z))\bigr)t^{\nu},
\]
which equals
\[
\mathrm{Y}=\frac{t^9(1+t^{12})}{(1-t^4)(1-t^6)(1-t^{12})},
\]
and should be the generating series of the numbers of Yoshida lifts for $\Delta_K = -7,-8,-11$.

The generating series of dimensions of algebraic modular forms for $\SO(6)$ can be computed using Molien series as discussed earlier. The results are as follows:

\begin{enumerate}

\item For $\Delta_K=-3$, we have
\begin{align*}
\sum_{\nu=0}^{\infty}\dim \fM_{\nu}(\cL_K)t^{\nu} & =
\frac{(1+t^{14})(1+t^{36})}{(1-t^6)(1-t^8)(1-t^{10})(1-t^{12})(1-t^{18})};
\\
\sum_{\nu=0}^{\infty}\dim \fM_{\nu}(\cL_K,\spin_3)t^{\nu}& =
\frac{t^5 (1 + t^4)(1 + t^{36})}{(1-t^6)(1-t^8)(1-t^{10})(1-t^{12})
(1-t^{18})}.
\end{align*}

\item For $\Delta_K=-4$, we have 
\begin{align*}
\sum_{\nu=0}^{\infty}\dim\fM_{\nu}(\cL_K)t^{\nu}=
\frac{1+t^{36}}{(1-t^4)(1-t^6)(1-t^8)(1-t^{10})(1-t^{12})};\\
\sum_{\nu=0}^{\infty}\dim \fM_{\nu}(\cL_K,\spin_2)t^{\nu}
=
\frac{t^6+t^{30}}{(1-t^4)(1-t^6)(1-t^8)(1-t^{10})(1-t^{12})}.
\end{align*}

\item For $\Delta_K=-7$, we have 
\begin{align*}
\sum_{\nu=0}^{\infty}\dim\fM_{\nu}(\cL_K)t^{\nu} & =
\frac{1+t^8+t^{10}+t^{12}+t^{24}+t^{26}+t^{28}+t^{36}}{(1-t^4)(1-t^6)^2(1-t^{10})(1-t^{14})};
\\ 
\sum_{\nu=0}^{\infty}\dim \fM_{\nu}(\cL_K,\spin_7)t^{\nu}& =\frac{t^3+t^5+t^7+t^{15}+t^{21}+t^{29}+t^{31}
+t^{33}}{(1-t^4)(1-t^6)^2(1-t^{10})(1-t^{14})}.
\end{align*}

\item For $\Delta_K=-8$, we have 
\begin{align*}
\sum_{\nu=0}^{\infty}\dim \fM_{\nu}(\cL_K)t^{\nu} & =
\frac{1 + t^{26}}{(1 - t^2)(1 -t^4)(1 - t^6)(1 - t^8)(1 - t^{10})};
\\
\sum_{\nu=0}^{\infty}\dim \fM_{\nu}(\cL_K,\spin_2)t^{\nu}& =\frac{t^5+t^{21}}{(1 - t^2)(1 -t^4)(1 - t^6)(1 - t^8)(1 - t^{10})}.
\end{align*}

\item For $\Delta_K=-11$, we have 
\begin{align*}
\sum_{\nu=0}^{\infty}\dim\fM_{\nu}(\cL_K)t^{\nu} & =
\frac{(1 + t^6)(1 + t^{20})}{(1 - t^2)(1 - t^4)(1 - t^6)(1 - t^8)(1 -  t^{10})};
\\
\sum_{\nu=0}^{\infty}\dim \fM_{\nu}(\cL_K,\spin_{11})t^{\nu}& =\frac{t(1+t^4)(1 + t^{20})}{(1 - t^2)(1 - t^4)(1 - t^6)(1 - t^8)(1 - t^{10})}.
\end{align*}

\end{enumerate}

\begin{rem} The dimension formulas for discriminant $\Delta_K = -3,-4,-7,-8$ can be interpreted in terms of the Weyl groups of semisimple Lie algebras. The genera of rank six lattices in these cases consist of a single class, represented by the root lattices of the Lie algebras of type $E_6$, $C_6$, $A_6$ and $C_5 \oplus A_1$, respectively, and the orthogonal group is the Weyl groups of the algebra. An algebraic modular form (for the group $\mathrm{O}(6)$) is the same as a Weyl-invariant harmonic polynomial. Here the root lattice of type $C_n$ is the same as the root lattice of type $D_n$, but the Weyl groups $W(C_n)$ and $W(D_n)$ are different. The Weyl-invariant polynomials form a polynomial ring that is given explicitly by the Harish--Chandra isomorphism. This explains the dimension formulas and it is also a useful point of view for calculations of eigenforms and their $L$-functions. \\

For example, the fundamental invariants of the Weyl group $W(E_6)$ are polynomials of degree $k = 2,5,6,8,9,12$. They can be chosen to have the form $$p_k = \sum_{i=1}^{27} \ell_i^k, \quad k \in \{2,5,6,8,9,12\}$$ where $\ell_i$ are 27 linear forms (essentially the 27 lines on a cubic surface in $\mathbb{P}^3$), see \cite{Mehta1988}. On the other hand, the polynomials of degrees $5,6,8,9,12$ can be chosen to be harmonic, such that they define algebraic modular forms for $\mathrm{O}(6)$ (with spinor character $\mathrm{spin}_3$ if the degree is odd, and trivial character otherwise). To obtain an algebraic modular form with the determinant character on $\mathrm{O}(6)$, one can take the Jacobian determinant $$\Psi_{36} = \mathrm{det} \Big( \nabla p_2, \nabla p_5, \nabla p_6, \nabla p_8, \nabla p_9, \nabla p_{12} \Big)$$ which has degree $36$. This has minimal degree among polynomials that are invariant up to the determinant character, so it is automatically harmonic. \\
The spaces of algebraic modular forms in fact have a graded ring structure, with multiplication defined by $$f \times g := \pi_{\mathrm{har}}(fg)$$ where $\pi_{\mathrm{har}}$ is the harmonic projection (i.e. $\pi_{\mathrm{har}}(f)$ is the unique harmonic polynomial for which $f - \pi_{\mathrm{har}}(f)$ is a multiple of the quadratic form $\langle x, x \rangle$). Since $p_5, p_6, p_8, p_9, p_{12}$ and $\langle -, - \rangle$ (which is essentially $p_2$) are algebraically independent, it follows that $p_5,p_6,p_8,p_9,p_{12}$ freely generate the \emph{ring} $\mathfrak{M}_*(\mathrm{O}(E_6), \mathrm{spin}_*)$, and that $\Psi_{36}$ generates the module of forms with the determinant character. From this we obtain $$\sum_{\nu=0}^{\infty} \mathrm{dim}\, \mathfrak{M}_{\nu}(\Spin(\cL_K)) t^{\nu} = \frac{1 + t^{36}}{(1 - t^5)(1 - t^6)(1 - t^8)(1 - t^9)(1 - t^{12})}.$$
The computation for the other root lattices is similar.
\end{rem}

\begin{rem}
    Conjecture \ref{conj:main} predicts that the Jacobian determinant $\Psi_{36}$ corresponds to a Hermitian cusp form eigenform of weight 40 for the field $K = \mathbb{Q}(\sqrt{-3})$, with all Hecke eigenvalues in $\mathbb{Q}$. By computing Hermitian modular forms for $\mathbb{Q}(\sqrt{-3})$ in terms of the generators described in \cite{DK1}, we were able to split the space of cusp forms of weight 40 into eigenforms and we found that there really is a rational eigenform in weight 40. The beginning of its Fourier series is available on the second author's webpage \cite{github}; a remarkable number of its Fourier coefficients turn out to vanish.
\end{rem}

\begin{thrm}
Let $\Delta_K \in \{-3, -4, -7, -8, -11\}$.
The dimensions of modular forms are related as follows:
\begin{align*}
& \sum_{\nu=0}^{\infty}
\dim M_{\nu+4}(\Gamma_K)t^{\nu+4}-\Eis-\KE-\MW
\\ & =\sum_{\nu=0}^{\infty}(\dim \fM_{\nu}(\Spin(\cL_K)))t^{\nu+4}
-\dim S_3(\Gamma_0(|\Delta_K|),\chi_k)\times \mathrm{Y}\cdot t^4,
\end{align*}
and
\begin{align*} \sum_{\nu=0}^{\infty}
\dim M_{\nu+4}(\Gamma_K^*)t^{\nu+4}-\Eis-\KE-\MW
=\sum_{\nu=0}^{\infty}\dim \fM_{\nu}(\cL_K)t^{\nu+4}.
\end{align*}

\end{thrm}

Remark. We have $\dim S_3(\Gamma_0(|\Delta_K|),\chi_K)=1$ if 
$\Delta_K=-7$, $-8$, $-11$ and $\dim S_3(\Gamma_0(|\Delta_K|),\chi_K)=0$ if $\Delta_K=-3$, $-4$.

\bibliographystyle{plainnat}
\bibliofont

\appendix

\section{The dimension formula for \texorpdfstring{$\mathbb{Q}(\sqrt{-2})$}{QQ(sqrt -2)}}

A set of generators for the graded ring of Hermitian modular forms over $\mathbb{Q}(\sqrt{-2})$ was computed by Dern and Krieg \cite{DK2}, but they do not give a formula for the dimensions of spaces of modular forms and it does not seem to be easy to derive it from their work. However, the generating series for $\mathrm{dim}\, M_k(\Gamma_K)$ can be computed directly by means of Jacobi forms. This is an application of the argument of section 6 of \cite{WW2024} which computes dimensions of spaces of modular forms associated to reducible root lattices. \\

The Fourier--Jacobi expansion of an arbitrary modular form for $\Gamma_K$ takes the form $$F\Big( \begin{pmatrix} \tau & z_1 \\ z_2 & w \end{pmatrix} \Big) = \sum_{n=0}^{\infty} f_n(\tau, z_1, z_2) e^{2\pi i n w},$$ where each $f_n$ is a Hermitian Jacobi form of index $n$, which is essentially the same (up to a change of variables) as a Jacobi form of index $n$ on the lattice $\mathcal{O}_K$  (cf. \cite{Haverkamp96}); more precisely, the functions $$f_n(\tau, w_1, w_2) := f_n\Big( \tau, \frac{z_1 + z_2}{2}, \frac{z_1 - z_2}{2i} \Big)$$ 
are Jacobi forms of lattice index. By abuse of notation we denote the latter Jacobi form by $f_n$ also. For $K = \mathbb{Q}(\sqrt{-2})$ we have an isometry $\mathcal{O}_K \cong A_1 \oplus A_1(2) =: L$. \\

Note that if $F(Z^T) = \varepsilon (-1)^k F(Z)$ with $\varepsilon \in \{\pm 1\}$ then the transformations under $Z \mapsto Z^T$ and under the M\"obius transformation $M = \mathrm{diag}(1,-1,1,-1) \in \Gamma_K$ yield $$f_n(\tau, w_1, -w_2) = \varepsilon (-1)^k f_n(\tau, w_1, w_2), \quad f_n(\tau, -w_1, w_2) = \varepsilon \cdot f_n(\tau, w_1, w_2).$$

If $f$ is symmetric, then the Jacobi forms $f_n$ are invariant under the Weyl group $W(L)$; moreover the Fourier series of the first nonzero term $f_N$ vanishes to $q$-order $N$.  By Theorem 2.4 of \cite{WW2023} it can be written uniquely as a $\mathbb{C}[E_4,E_6]$-linear combination of forms $$f_N(\tau, w_1, w_2) = \Delta^N(\tau) \psi_1(\tau, w_1) \psi_2(\tau, w_2),$$ where $\psi_1$ is an even weak Jacobi form of index $N$, and where $\psi_2$ is a weak Jacobi form for the lattice $A_1(2)$ and index $N$ (or equivalently, a weak Jacobi form in the usual sense of index $2N$), which are both monomials in the basic weak Jacobi forms $\phi_{-2, 1}$ and $\phi_{0, 1}$ of index one and the basic weak Jacobi form $\phi_{-1, 2}$ of index two. (For the ring structure of weak Jacobi forms and the three generators $\phi_{k, m}$ see Chapter 9 of \cite{EZ}.)

Conversely every such Jacobi form arises as the leading term of the Fourier--Jacobi series of a holomorphic Hermitian modular form $F$: one can construct $F$ as a product of Gritsenko--Maass lifts. In the terminology of \cite{WW2024}, this is possible because the lattice $L$ satisfies the $\mathrm{Norm}_2$-condition. \\

If on the other hand $f$ is skew-symmetric, then the leading Fourier--Jacobi coefficient $f_N$ is odd with respect to its first elliptic variable and it vanishes to $q$-order at least $N+1$, so can be expressed uniquely as a $\mathbb{C}[E_4,E_6]$-linear combination of Jacobi forms $$f_N(\tau, w_1, w_2) = \Delta^{N+1}(\tau) \phi_{-1, 2}(\tau, w_1) \psi_1(\tau, w_1) \psi_2(\tau, w_2),$$ where $\psi_1$ and $\psi_2$ are monomials in $\phi_{-2, 1}$ and $\phi_{0, 1}$ of index $N-2$ and $2N$, respectively, and where $\phi_{-1, 2}$ is the weak Jacobi form of weight $(-1)$ and index 2 defined in \cite{EZ}. Conversely, one can construct holomorphic Hermitian modular forms with any such leading coefficient $f_N$ by multiplying Borcherds products with certain singular Gritsenko lifts, or quotients of holomorphic Hermitian modular forms; this was done in \cite{DK2}. \\

Note that $$\sum_{k=0}^{\infty} \sum_{m=0}^{\infty} \mathrm{dim}\, J_{k, m}^{\mathrm{weak}} \, t^k x^m = \frac{1 + x^2 / t}{(1 - t^4)(1 - t^6)(1 - x)(1 - x / t^2)}, \quad |x| < |t|^2 < 1.$$

From the above observations we have
\begin{align*} &\quad \frac{1}{(1 - t^4)(1 - t^6)} \sum_{k=0}^{\infty} \mathrm{dim} \, M_k^{\mathrm{sym}}(\Gamma_K) t^k \\ &= \sum_{N=0}^{\infty} \sum_{a=0}^{\infty} \sum_{b=0} ^{\infty} \Big( \mathrm{dim}\, J_{a, N}^{\mathrm{weak}, \mathrm{even}} \Big) \cdot \Big( \mathrm{dim}\, J_{b, 2N}^{\mathrm{weak}} \Big) t^{a+b+12N} \\ &= \frac{1}{2\pi i} \oint \Big( \sum_{a=0}^{\infty} \sum_{m=0}^{\infty} \mathrm{dim}\, J_{a, m}^{\mathrm{weak}, \mathrm{even}} t^a x^{2m} \Big) \Big( \sum_{b=0}^{\infty} \sum_{n=0}^{\infty} \mathrm{dim}\, J_{b, n}^{\mathrm{weak}} t^{b + 6n} x^{-n} \Big) \, \frac{\mathrm{d}x}{x} \\ &= \frac{1}{(1 - t^4)^2 (1 - t^6)^2} \cdot \frac{1}{2\pi i}  \oint \frac{(1 + x^{-2} t^{11})}{(1 - x^2)(1 - t^6/x)(1 - x^2 / t^2)(1 - t^4 / x)} \, \frac{\mathrm{d}x}{x} \\ &=: \frac{1}{(1 - t^4)^2 (1 - t^6)^2} \cdot \frac{1}{2\pi i} \oint \, \omega. \end{align*} 

For fixed $t$, we integrate along a circle in the $x$-plane centered at $0$ of any radius $\varepsilon$ for which $|t|^4 < \varepsilon < |t|$ (i.e. within the annulus on which the integrand converges as a Laurent series). Therefore only the singularities at $x \in \{0,t^4,t^6\}$ contribute to the dimension integral. We have $$\mathrm{Res}_{x=0} (\omega) = t, \quad \mathrm{Res}_{x=t^4}(\omega) = \frac{1}{(1 - t^2)(1 - t^3)(1 - t^8)}, \quad \mathrm{Res}_{x=t^6}(\omega) = -\frac{t}{(1 - t)(1 - t^{10})(1 - t^{12})},$$ hence $$\sum_{k=0}^{\infty} \mathrm{dim} \, M_k^{\mathrm{sym}}(\Gamma_K) t^k = \frac{1}{(1 - t^4)(1 - t^6)} \Big[ t + \frac{1}{(1 - t^2)(1 - t^3)(1 - t^8)} - \frac{t}{(1 - t)(1 - t^{10})(1 - t^{12})} \Big].$$

By an analogous computation, \begin{align*} &\quad \frac{1}{(1 - t^4)(1 - t^6)} \sum_{k=0}^{\infty} \mathrm{dim}\, M_k^{\mathrm{skew}}(\Gamma_K) t^k \\ &= \sum_{N=0}^{\infty} \sum_{a=0}^{\infty} \sum_{b=0}^{\infty} \Big( \mathrm{dim}\, J_{a, N}^{\mathrm{weak}, \mathrm{odd}} \Big) \cdot \Big( \mathrm{dim}\, J_{b, 2N}^{\mathrm{weak}} \Big) t^{a+b+12N+12} \\ &= \frac{t^{12}}{(1 - t^4)^2(1 - t^6)^2} \cdot \frac{1}{2\pi i} \oint \frac{x^4 t^{-1} (1 + x^{-2} t^{11})}{(1 - x^2)(1 - t^6 / x)(1 - x^2 / t^2)(1 - t^4 / x)} \, \frac{\mathrm{d}x}{x}.\end{align*}

The integrand now has no pole at $x=0$ and, compared with the symmetric case, the residues at the poles $x=t^4$ and $x = t^6$ are multiplied by $t^{15}$ and $t^{23}$, respectively. Therefore $$\sum_{k=0}^{\infty} \mathrm{dim} \, M_k^{\mathrm{skew}}(\Gamma_K) t^k = \frac{t^{12}}{(1 - t^4)(1 - t^6)} \Big[ \frac{t^{15}}{(1 - t^2)(1 - t^3)(1 - t^8)} - \frac{t^{24}}{(1 - t)(1 - t^{10})(1 - t^{12})} \Big].$$

Altogether we have $$\sum_{k=0}^{\infty} \mathrm{dim} \, M_k(\Gamma_K) t^k = \frac{1}{(1 - t^4)(1 - t^6)} \Big[ t + \frac{1 + t^{27}}{(1 - t^2)(1 - t^3)(1 - t^8)} - \frac{t + t^{36}}{(1 - t)(1 - t^{10})(1 - t^{12})} \Big].$$
\section{Tables of Hermitian eigenforms}

In the appendix, we describe the decomposition of $M_k(\Gamma_K)$, $K = \mathbb{Q}(\sqrt{\Delta})$ for the five smallest discriminants $\Delta$ as predicted by Conjecture \ref{conj:mainconj}. In particular, we work out the dimensions of the subspaces of Miyawaki lifts and those which should correspond to algebraic modular forms on $\SO(6)$ twisted by characters.

Hermitian eigenforms that are not cusp forms are of two types: Klingen Eisenstein series, which are induced from cusp forms on $\SL_2(\mathbb{Z})$, and the true Eisenstein series, which is the Gritsenko--Maass lift of a Jacobi Eisenstein series (of lattice index $\mathcal{O}_K$). The dimensions of these spaces are $$\mathrm{dim}\, M_k^{\mathrm{Klingen}} = \mathrm{dim}\, S_k(\SL_2(\mathbb{Z})),$$ and $\mathrm{dim}\, M_k^{\mathrm{Eis}} = 1$ (if $k \ge 4$ is even) or $\mathrm{dim}\, M_k^{\mathrm{Eis}} = 0$ otherwise. For Eisenstein series (of either type) on unitary groups see \cite{Shimura97}.

$\Maass(\chi)$ stands for Hermitian modular forms that (conjecturally) correspond to eigenforms on $\SO(6)$ with nonzero image under the theta map and spin character $\chi$. (If $\chi = 1$ then we omit it.) The character $\chi$ equivalently describes the eigenvalues under the Atkin--Lehner involutions $W_d$. By a theorem of Wernz \cite{Wernz2020}, these forms span the Sugano Maass space as $\chi$ runs through spinor characters and the subspace with trivial spin character is exactly Krieg's Maass space. 
An explicit dimension formula for the Maass space was given by Sugano \cite{sugano}; see also Haverkamp \cite{Haverkamp96}.

``$\MW$" is spanned by Miyawaki lifts, or those eigenforms whose degree six $L$-functions factor as $$L(F, s) = \zeta_K(s - k + 2) \cdot L(f \otimes g; s),$$ where $f \in S_k(\SL_2(\mathbb{Z}))$ and $g \in S_{k-2}(\SL_2(\mathbb{Z}))$.

$\G_k(\chi)$ (for ``general") stands for non-Maass cuspidal Hermitian eigenforms which (conjecturally) correspond to algebraic modular forms on $\SO(6)$ of weight $(k-4)$ (twisted by $\chi$, a product of spinor characters and possibly $\mathrm{det}$) that belong to the kernel of the theta map. These are expected to be orthogonal to all Maass and Miyawaki lifts.

$\color{red} \mathrm{Y}_k$ represents (conjectural) Yoshida lifts of weight $(k-4)$, which are automorphic forms on $\SO(6)$ that should not have an associated Hermitian modular form (and therefore do not contribute to the total dimension). These forms are listed only for completeness.

In all cases, we found that the conjectural decomposition coincides exactly with the factorization of the minimal polynomial of a single Hecke operator $T_p$ over $\mathbb{Q}$. (Namely the operator $T_2$ if $2 \nmid \Delta_K$ and $T_3$ otherwise. We use the smallest $p$ with $p \nmid \Delta_K$ because its action on modular forms can be computed with the smallest number of Fourier coefficients.)

\begin{figure}

\begin{table}[H]
\begin{tabular}{|l|l|l|l|l|l|l|l|l|l|l|l|l|l|l|l|l|l|l|l|l|}
\hline
                                                                  & 1 & 2 & 3 & 4 & 5 & 6 & 7 & 8 & 9 & 10 & 11 & 12 & 13 & 14 & 15 & 16 & 17 & 18 & 19 & 20 \\ \hline
Eisenstein                                                        & 0 & 0 & 0 & 1 & 0 & 1 & 0 & 1 & 0 & 1  & 0  & 1  & 0  & 1  & 0  & 1  & 0  & 1 & 0 & 1 \\ \hline
Klingen                                                & 0 & 0 & 0 & 0 & 0 & 0 & 0 & 0 & 0 & 0  & 0  & 1  & 0  & 0  & 0  & 1  & 0  & 1  & 0 & 1 \\ \hline
Maass                                                             & 0 & 0 & 0 & 0 & 0 & 0 & 0 & 0 & 0 & 1  & 0  & 1  & 0  & 1  & 0  & 2  & 0  & 2 & 0 & 2 \\ \hline
$\Maass(\spin_3)$              & 0 & 0 & 0 & 0 & 0 & 0 & 0 & 0 & 1 & 0  & 0  & 0  & 1  & 0  & 1  & 0  & 1  & 0 & 1 & 0 \\ \hline
$\MW$                                                            & 0 & 0 & 0 & 0 & 0 & 0 & 0 & 0 & 0 & 0  & 0  & 0  & 0  & 0  & 0  & 0  & 0  & 1 & 0 & 1 \\ \hline
$\G$                                              & 0 & 0 & 0 & 0 & 0 & 0 & 0 & 0 & 0 & 0  & 0  & 0  & 0  & 0  & 0  & 0  & 0  & 0 & 0 & 0 \\ \hline
$\G(\spin_3)$                      & 0 & 0 & 0 & 0 & 0 & 0 & 0 & 0 & 0 & 0  & 0  & 0  & 0  & 0  & 0  & 0  & 0  & 0 & 1 & 0 \\ \hline
$\G(\mathrm{det})$                         & 0 & 0 & 0 & 0 & 0 & 0 & 0 & 0 & 0 & 0  & 0  & 0  & 0  & 0  & 0  & 0  & 0  & 0 & 0 & 0 \\ \hline
$\G(\spin_3 \otimes \mathrm{det})$ & 0 & 0 & 0 & 0 & 0 & 0 & 0 & 0 & 0 & 0  & 0  & 0  & 0  & 0  & 0  & 0  & 0  & 0 & 0 & 0 \\ \hline
Total                                                   & 0 & 0 & 0 & 1 & 0 & 1 & 0 & 1 & 1 & 2  & 0  & 3  & 1  & 2  & 1  & 4  & 1  & 5 & 2 & 5 \\ \hline
\end{tabular}
\end{table}

\begin{small}
\begin{table}[H]
\begin{tabular}{|l|l|l|l|l|l|l|l|l|l|l|l|l|l|l|l|l|l|l|l|l|}
\hline
                                            & 21 & 22 & 23 & 24 & 25 & 26 & 27 & 28 & 29 & 30 & 31 & 32 & 33 & 34 & 35 & 36 & 37 & 38 & 39 & 40 \\ \hline
Eisenstein                                  & 0  & 1  & 0  & 1  & 0  & 1  & 0  & 1  & 0  & 1  & 0  & 1  & 0  & 1  & 0  & 1 & 0 & 1 & 0 & 1  \\ \hline
Klingen                        & 0  & 1  & 0  & 2  & 0  & 1  & 0  & 2  & 0  & 2  & 0  & 2  & 0  & 2  & 0  & 3 & 0 & 2 & 0 & 3  \\ \hline
Maass                                        & 0  & 3  & 0  & 3  & 0  & 3  & 0  & 4  & 0  & 4  & 0  & 4  & 0  & 5  & 0  & 5 & 0 & 5 & 0 & 6 \\ \hline
$\Maass(\spin_3)$             & 2  & 0  & 1  & 0  & 2  & 0  & 2  & 0  & 2  & 0  & 2  & 0  & 3  & 0  & 2  & 0 & 3 & 0 & 3 & 0 \\ \hline
$\MW$                                        & 0  & 1  & 0  & 2  & 0  & 2  & 0  & 2  & 0  & 4  & 0  & 4  & 0  & 4  & 0  & 6 & 0 & 6 & 0 & 6 \\ \hline
$\G$                                       & 0  & 1  & 0  & 1  & 0  & 1  & 0  & 3  & 0  & 3  & 0  & 3  & 6  & 6  & 0  & 7 & 0 & 7 & 0 & 11 \\ \hline
$\G(\spin_3)$                        & 1  & 0  & 1  & 0  & 2  & 0  & 3  & 0  & 3  & 0  & 5  & 0  & 0  & 0  & 6  & 0 & 9 & 0 & 11 & 0 \\ \hline
$\G(\mathrm{det})$                           & 0  & 0  & 0  & 0  & 0  & 0  & 0  & 0  & 0  & 0  & 0  & 0  & 0  & 0  & 0  & 0 & 0 & 0 & 0 & 1 \\ \hline
$\G(\spin_3 \otimes \mathrm{det})$  & 0  & 0  & 0  & 0  & 0  & 0  & 0  & 0  & 0  & 0  & 0  & 0  & 0  & 0  & 0  & 0 & 0 & 0 & 0 & 0 \\ \hline
Total                              & 3  & 7  & 2  & 9  & 4  & 8  & 5  & 12 & 5  & 14 & 7  & 14 & 9  & 18 & 8  & 22 & 12 & 21 & 14 & 28 \\ \hline
\end{tabular}
\end{table}

\end{small}

\caption{Hermitian eigenforms for discriminant $-3$. Note that the first eigenform with spinor character $\spin_3 \otimes \mathrm{det}$ occurs in weight 45 and therefore does not appear in the table.}

\end{figure}

\begin{figure}
\begin{small}

\begin{table}[H]
\begin{tabular}{|l|l|l|l|l|l|l|l|l|l|l|l|l|l|l|l|l|l|l|l|l|}
\hline
                                                                  & 1 & 2 & 3 & 4 & 5 & 6 & 7 & 8 & 9 & 10 & 11 & 12 & 13 & 14 & 15 & 16 & 17 & 18 & 19 & 20 \\ \hline
Eisenstein                                                        & 0 & 0 & 0 & 1 & 0 & 1 & 0 & 1 & 0 & 1  & 0  & 1  & 0  & 1  & 0  & 1  & 0  & 1 & 0 & 1 \\ \hline
Klingen                                                & 0 & 0 & 0 & 0 & 0 & 0 & 0 & 0 & 0 & 0  & 0  & 1  & 0  & 0  & 0  & 1  & 0  & 1  & 0 & 1 \\ \hline
Maass                                                             & 0 & 0 & 0 & 0 & 0 & 0 & 0 & 1 & 0 & 1  & 0  & 2  & 0  & 2  & 0  & 3  & 0  & 3 & 0 & 4 \\ \hline
$\Maass(\spin_2)$              & 0 & 0 & 0 & 0 & 0 & 0 & 0 & 0 & 0 & 1  & 0  & 0  & 0  & 1  & 0  & 1  & 0  & 1 & 0 & 1 \\ \hline
$\MW$                                                            & 0 & 0 & 0 & 0 & 0 & 0 & 0 & 0 & 0 & 0  & 0  & 0  & 0  & 0  & 0  & 0  & 0  & 1 & 0 & 1 \\ \hline
$\G$                                              & 0 & 0 & 0 & 0 & 0 & 0 & 0 & 0 & 0 & 0  & 0  & 0  & 0  & 0  & 0  & 1  & 0  & 0 & 0 & 2 \\ \hline
$\G(\spin_2)$                      & 0 & 0 & 0 & 0 & 0 & 0 & 0 & 0 & 0 & 0  & 0  & 0  & 0  & 0  & 0  & 0  & 0  & 1 & 0 & 1 \\ \hline
$\G(\mathrm{det})$                         & 0 & 0 & 0 & 0 & 0 & 0 & 0 & 0 & 0 & 0  & 0  & 0  & 0  & 0  & 0  & 0  & 0  & 0 & 0 & 0 \\ \hline
$\G(\spin_2 \otimes \mathrm{det})$ & 0 & 0 & 0 & 0 & 0 & 0 & 0 & 0 & 0 & 0  & 0  & 0  & 0  & 0  & 0  & 0  & 0  & 0 & 0 & 0 \\ \hline
Total                                                   & 0 & 0 & 0 & 1 & 0 & 1 & 0 & 2 & 0 & 3  & 0  & 4  & 0  & 4  & 0  & 7  & 0  & 8 & 0 & 11 \\ \hline
\end{tabular}
\end{table}

\begin{table}[H]
\begin{tabular}{|l|l|l|l|l|l|l|l|l|l|l|l|l|l|l|l|l|l|l|l|l|}
\hline
                                                                  & 21 & 22 &23 & 24 & 25 & 26 & 27 & 28 & 29 & 30 & 31 & 32 & 33 & 34 & 35 & 36 & 37 & 38 & 39 & 40 \\ \hline
Eisenstein                                                        & 0 & 1 & 0 & 1 & 0 & 1 & 0 & 1 & 0 & 1  & 0  & 1  & 0  & 1  & 0  & 1  & 0  & 1 & 0 & 1 \\ \hline
Klingen                                                & 0  & 1  & 0  & 2  & 0  & 1  & 0  & 2  & 0  & 2  & 0  & 2  & 0  & 2  & 0  & 3 & 0 & 2 & 0 & 3  \\ \hline
Maass                                                             & 0 & 4 & 0 & 5 & 0 & 5 & 0 & 6 & 0 & 6  & 0  & 7  & 0  & 7  & 0  & 8  & 0  & 8 & 0 & 9 \\ \hline
$\Maass(\spin_2)$              & 0 & 2 & 0 & 1 & 0 & 2 & 0 & 2 & 0 & 2  & 0  & 2  & 0  & 3  & 0  & 2  & 0  & 3 & 0 & 3 \\ \hline
$\MW$                                                            & 0  & 1  & 0  & 2  & 0  & 2  & 0  & 2  & 0  & 4  & 0  & 4  & 0  & 4  & 0  & 6 & 0 & 6 & 0 & 6 \\ \hline
$\G$                                              & 0 & 2 & 0 & 4 & 0 & 4 & 0 & 8 & 0 & 7  & 0  & 12  & 0  & 13  & 0  & 18  & 0  & 19 & 0 & 27 \\ \hline
$\G(\spin_2)$                      & 0 & 2 & 0 & 2 & 0 & 4 & 0 & 4 & 0 & 7  & 0  & 7  & 0  & 11  & 0  & 11  & 0  & 16 & 0 & 17 \\ \hline
$\G(\mathrm{det})$                         & 0 & 0 & 0 & 0 & 0 & 0 & 0 & 0 & 0 & 0  & 0  & 0  & 0  & 0  & 0  & 0  & 0  & 0 & 0 & 1 \\ \hline
$\G(\spin_2 \otimes \mathrm{det})$ & 0 & 0 & 0 & 0 & 0 & 0 & 0 & 0 & 0 & 0  & 0  & 0  & 0  & 1  & 0  & 0  & 0  & 1 & 0 & 1 \\ \hline
Total                                                   & 0 & 13 & 0 & 17 & 0 & 19 & 0 & 25 & 0 & 29  & 0  & 35  & 0  & 42  & 0  & 49  & 0  & 56 & 0 & 68 \\ \hline
\end{tabular}
\end{table}

\caption{Hermitian eigenforms for discriminant $-4$}
\end{small}
\end{figure}

\begin{figure}
    
\begin{table}[H]
\begin{tabular}{|l|l|l|l|l|l|l|l|l|l|l|l|l|l|l|l|l|l|l|l|l|}
\hline
                                                                  & 1 & 2 & 3 & 4 & 5 & 6 & 7 & 8 & 9 & 10 & 11 & 12 & 13 & 14 & 15 & 16 & 17 & 18 & 19 & 20 \\ \hline
Eisenstein                                                        & 0 & 0 & 0 & 1 & 0 & 1 & 0 & 1 & 0 & 1  & 0  & 1  & 0  & 1  & 0  & 1  & 0  & 1 & 0 & 1 \\ \hline
Klingen                                                & 0 & 0 & 0 & 0 & 0 & 0 & 0 & 0 & 0 & 0  & 0  & 1  & 0  & 0  & 0  & 1  & 0  & 1 & 0 & 1 \\ \hline
Maass                                                             & 0 & 0 & 0 & 0 & 0 & 0 & 0 & 1 & 0 & 2  & 0  & 2  & 0  & 3  & 0  & 4  & 0  & 4  & 0 & 5\\ \hline
$\Maass(\spin_7)$              & 0 & 0 & 0 & 0 & 0 & 0 & 1 & 0 & 1 & 0  & 2  & 0  & 2  & 0  & 3  & 0  & 3  & 0 & 4 & 0  \\ \hline
$\MW$                                                            & 0 & 0 & 0 & 0 & 0 & 0 & 0 & 0 & 0 & 0  & 0  & 0  & 0  & 0  & 0  & 0  & 0  & 1 & 0 & 1 \\ \hline
$\G$                                              & 0 & 0 & 0 & 0 & 0 & 0 & 0 & 0 & 0 & 0  & 0  & 0  & 0  & 1  & 0  & 2  & 0  & 3 & 0 & 5  \\ \hline
$\G(\spin_7)$                      & 0 & 0 & 0 & 0 & 0 & 0 & 0 & 0 & 0 & 0  & 0  & 0  & 0  & 0  & 1  & 0  & 2  & 0 & 4 & 0 \\ \hline
$\G(\mathrm{det})$                         & 0 & 0 & 0 & 0 & 0 & 0 & 0 & 0 & 0 & 0  & 0  & 0  & 0  & 0  & 0  & 0  & 0  & 0 & 0 & 0 \\ \hline
$\G(\spin_7 \otimes \mathrm{det})$ & 0 & 0 & 0 & 0 & 0 & 0 & 0 & 0 & 0 & 0  & 0  & 0  & 0  & 0  & 0  & 0  & 0  & 0 & 0 & 0 \\ \hline
$\color{red}\mathrm{Y}$ 				       & 0 & 0 & 0 & 0 & 0 & 0 & 0 & 0 & 0 & 0  & 0  & 0  & $\color{red} 1$  & 0  & 0  & 0  & $\color{red} 1$  & 0 & $\color{red} 1$ & 0 \\ \hline
Total                                                   & 0 & 0 & 0 & 1 & 0 & 1 & 1 & 2 & 1 & 3  & 2  & 4  & 2  & 5  & 4  & 8  & 5  & 10 & 8 & 13 \\ \hline
\end{tabular}
\end{table}

\caption{Hermitian eigenforms for discriminant $-7$}

\end{figure}

\begin{figure}
\begin{table}[H]
\begin{tabular}{|l|l|l|l|l|l|l|l|l|l|l|l|l|l|l|l|l|l|l|l|l|}
\hline
                                                                  & 1 & 2 & 3 & 4 & 5 & 6 & 7 & 8 & 9 & 10 & 11 & 12 & 13 & 14 & 15 & 16 & 17 & 18 & 19 & 20 \\ \hline
Eisenstein                                                        & 0 & 0 & 0 & 1 & 0 & 1 & 0 & 1 & 0 & 1  & 0  & 1  & 0  & 1  & 0  & 1  & 0  & 1 & 0 & 1 \\ \hline
Klingen                                                & 0 & 0 & 0 & 0 & 0 & 0 & 0 & 0 & 0 & 0  & 0  & 1  & 0  & 0  & 0  & 1  & 0  & 1 & 0 & 1 \\ \hline
Maass                                                             & 0 & 0 & 0 & 0 & 0 & 1 & 0 & 2 & 0 & 3  & 0  & 4  & 0  & 5  & 0  & 6  & 0  & 7  & 0 & 8 \\ \hline
$\Maass(\spin_2)$              & 0 & 0 & 0 & 0 & 0 & 0 & 0 & 0 & 1 & 0  & 1  & 0  & 1  & 0  & 2  & 0  & 2  & 0 & 2 & 0  \\ \hline
$\MW$                                                            & 0 & 0 & 0 & 0 & 0 & 0 & 0 & 0 & 0 & 0  & 0  & 0  & 0  & 0  & 0  & 0  & 0  & 1 & 0 & 1 \\ \hline
$\G$                                              & 0 & 0 & 0 & 0 & 0 & 0 & 0 & 0 & 0 & 0  & 0  & 1  & 0  & 2  & 0  & 4  & 0  & 6 & 0 & 10  \\ \hline
$\G(\spin_2)$                      & 0 & 0 & 0 & 0 & 0 & 0 & 0 & 0 & 0 & 0  & 0  & 0  & 0  & 0  & 1  & 0  & 2  & 0 & 4 & 0 \\ \hline
$\G(\mathrm{det})$                         & 0 & 0 & 0 & 0 & 0 & 0 & 0 & 0 & 0 & 0  & 0  & 0  & 0  & 0  & 0  & 0  & 0  & 0 & 0 & 0 \\ \hline
$\G(\spin_2 \otimes \mathrm{det})$ & 0 & 0 & 0 & 0 & 0 & 0 & 0 & 0 & 0 & 0  & 0  & 0  & 0  & 0  & 0  & 0  & 0  & 0 & 0 & 0 \\ \hline
$\color{red}\mathrm{Y}$ 				       & 0 & 0 & 0 & 0 & 0 & 0 & 0 & 0 & 0 & 0  & 0  & 0  & $\color{red} 1$  & 0  & 0  & 0  & $\color{red} 1$  & 0 & $\color{red} 1$ & 0 \\ \hline
Total                                                   & 0 & 0 & 0 & 1 & 0 & 2 & 0 & 3 & 1 & 4  & 1  & 7  & 1  & 8  & 3  & 12  & 4  & 16 & 6 & 21 \\ \hline
\end{tabular}
\end{table}

\caption{Hermitian eigenforms for discriminant $-8$}

\end{figure}

\begin{figure}
\begin{table}[H]
\begin{tabular}{|l|l|l|l|l|l|l|l|l|l|l|l|l|l|l|l|l|l|l|l|l|}
\hline
                                                                  & 1 & 2 & 3 & 4 & 5 & 6 & 7 & 8 & 9 & 10 & 11 & 12 & 13 & 14 & 15 & 16 & 17 & 18 & 19 & 20 \\ \hline
Eisenstein                                                        & 0 & 0 & 0 & 1 & 0 & 1 & 0 & 1 & 0 & 1  & 0  & 1  & 0  & 1  & 0  & 1  & 0  & 1 & 0 & 1 \\ \hline
Klingen                                                & 0 & 0 & 0 & 0 & 0 & 0 & 0 & 0 & 0 & 0  & 0  & 1  & 0  & 0  & 0  & 1  & 0  & 1 & 0 & 1 \\ \hline
$\Maass$                                                            & 0 & 0 & 0 & 0 & 0 & 1 & 0 & 2 & 0 & 3  & 0  & 4  & 0  & 5  & 0  & 6  & 0  & 7  & 0 & 8 \\ \hline
$\Maass(\mathrm{spin}_{11})$              & 0 & 0 & 0 & 0 & 1 & 0 & 1 & 0 & 3 & 0  & 3  & 0  & 4  & 0  & 5  & 0  & 6  & 0 & 6 & 0  \\ \hline
$\MW$                                                            & 0 & 0 & 0 & 0 & 0 & 0 & 0 & 0 & 0 & 0  & 0  & 0  & 0  & 0  & 0  & 0  & 0  & 1 & 0 & 1 \\ \hline
$\G$                                              & 0 & 0 & 0 & 0 & 0 & 0 & 0 & 0 & 0 & 1  & 0  & 2  & 0  & 4  & 0  & 7  & 0  & 11 & 0 & 17  \\ \hline
$\G(\mathrm{spin}_{11})$                      & 0 & 0 & 0 & 0 & 0 & 0 & 0 & 0 & 0 & 0  & 1  & 0  & 2  & 0  & 5  & 0  & 8  & 0 & 13 & 0 \\ \hline
$\G(\mathrm{det})$                         & 0 & 0 & 0 & 0 & 0 & 0 & 0 & 0 & 0 & 0  & 0  & 0  & 0  & 0  & 0  & 0  & 0  & 0 & 0 & 0 \\ \hline
$\G(\mathrm{spin}_{11} \otimes \mathrm{det})$ & 0 & 0 & 0 & 0 & 0 & 0 & 0 & 0 & 0 & 0  & 0  & 0  & 0  & 0  & 0  & 0  & 0  & 0 & 0 & 0 \\ \hline
$\color{red}\mathrm{Y}$ 				       & 0 & 0 & 0 & 0 & 0 & 0 & 0 & 0 & 0 & 0  & 0  & 0  & $\color{red} 1$  & 0  & 0  & 0  & $\color{red} 1$  & 0 & $\color{red} 1$ & 0 \\ \hline
Total                                                   & 0 & 0 & 0 & 1 & 1 & 2 & 1 & 3 & 3 & 5  & 4  & 8  & 6  & 10  & 10  & 15  & 14  & 21 & 19 & 28 \\ \hline
\end{tabular}
\end{table}

\caption{Hermitian eigenforms for discriminant $-11$}

\end{figure}

\clearpage

\section{Modular forms of small weight}

Additional evidence for Conjecture \ref{conj:main} comes from comparing modular forms of low weight as the discriminant varies.

Hermitian modular forms of weight four that do not belong to the Maass space are relatively rare for small discriminants. Therefore, even though no general formula for $\mathrm{dim}\, M_4(\Gamma_K)$ is known, we expect it to be close to the dimension $\mathrm{dim}\, J_4(\mathcal{O}_4)$ of the space of Jacobi forms which lift to the Maass space as long as $|\Delta_K|$ is not too large.

The following table shows that $\mathrm{dim}\, J_4(\mathcal{O}_K)$ is indeed very close to $\dim\fM_0(\Spin(\cL_K))$:

\begin{figure}[htbp]
\centering
\begin{tabular}{l*{14}{c}r}
\hline
$-\Delta$ & 3 & 4 & 7 & 8 & 11 & 15 & 19 & 20 & 23 & 24 & 31 & 35 & 39 \\
\hline
$\mathrm{dim}\, J_4(\mathcal{O}_K)$ & 1 & 1 & 1 & 1 & 1 & 1 & 2 & 1 & 1 & 2 & 2 & 2 & 2\\
\hline
$\dim\fM_0(\Spin(\cL_K))$ & 1 & 1 & 1 & 1 & 1 & 1 & 2 & 1 & 1 & 2 & 2 & 2 & 2 \\
\hline
\hline
$-\Delta$ & 40 & 43 & 47 & 51 & 52 & 55 & 56 & 59 & 67 & 68 & 71 & 79 & 83 \\
\hline
$\mathrm{dim}\, J_4(\mathcal{O}_K)$ & 3 & 4 & 2 & 4 & 4 & 3 & 3 & 4 & 6 & 4 & 3 & 5 & 6 \\
\hline
$\dim\fM_0(\Spin(\cL_K))$ & 3 & 4 & 2 & 4 & 4 & 3 & 3 & 4 & 6 & 4 & 3 & 5 & 6 \\
\hline
\hline
$-D$ & 84 & 87 & 88 & 91 & 95 & 103 & 104 & 107 & 111 & 115 & 116 & 120 & 123 \\
\hline
$\mathrm{dim}\, J_4(\mathcal{O}_K)$ & 5 & 5 & 7 & 7 & 4 & 7 & 6 & 8 & 6 & 8 & 7 & 7 & 10 \\
\hline
$\dim\fM_0(\Spin(\cL_K))$ & 5 & 5 & 7 & 8 & 4 & 7 & 7 & 8 & 6 & 9 & 7 & 7 & 10 \\
\hline
\hline
$-D$ & 127 & 131 & 132 & 136 & 139 & 143 & 148 & 151 & 152 & 155 & 159 & 163 & 164 \\
\hline
$\mathrm{dim}\, J_4(\mathcal{O}_K)$ & 9 & 9 & 9 & 10 & 11 & 7 & 12 & 10 & 10 & 10 & 9 & 14 & 10  \\
\hline
$\dim\fM_0(\Spin(\cL_K))$ & 9 & 10 & 9 & 12 & 13 & 8 & 12 & 11 & 11 & 11 & 10 & 16 & 11 \\
\hline
\hline
\end{tabular}    

\caption{Dimensions of Hermitian-Jacobi forms of weight 4 and algebraic modular forms of weight 0}

\end{figure}

Note that $J_4(\mathcal{O}_K)$ and $\mathfrak{M}_0(\Spin(\cL_K))$ both include Eisenstein series. The spaces of cusp forms are always one dimension smaller.

The table shows that $\mathrm{dim}\, \mathfrak{M}_0(\Spin(\cL_K))$ differs from $\mathrm{dim}\, J_4(\mathcal{O}_K)$ only for $-\Delta = 91, 104, 115, 131,...$ In these cases, there is a unique algebraic modular eigenform in the kernel of the theta map. These match the conjecture, however: \\

(1) $-\Delta = 91$ factors as $7 \cdot 13$. For the algebraic eigenform $F$ in the kernel of the theta map, E. Assaf computed the Euler factors of the standard $L$-function at small primes $p \ne 7, 13$ for us. Setting $X = p^{-s}$, the Euler factors decompose as follows:

\begin{align*} L_2(F; X) &= (1 + 3X + 16X^2) (1 + 6X + 16X^2) (1 - 4X)(1 + 4X); \\ L_3(F; X) &= (1 + 9X + 81X^2)(1 - 81X^2)(1 - 9X)^2; \\ L_5(F; X) &= (1 - 39X + 625X^2)(1 - 625X^2)^2; \\ L_{11}(F; X) &= (1 - 198X + 14641X^2)(1 + 66X + 14641X^2) (1 - 121X)(1 + 121X); \\ L_{17}(F; X) &= (1 + 569X + 83521X^2)(1 - 83521X^2)(1 - 289X)^2. \end{align*}

Besides the factors $(1 - p^2 X)$ and $(1 - \chi_7(p) p^2 X)$, where $\chi_7(n) = \left( \frac{n}{7} \right)$ is the quadratic residue symbol, there are factors $$1 - p a_p X + p^4 X^2, \quad 1 - b_p X + p^4 X^2,$$ where $a_p$ is the $p$th Fourier coefficient of $$\eta^3(\tau) \eta^3(7\tau) = q - 3q^2 + 5q^4 - 7q^7 - 3q^8 + 9q^9 - 6q^{11} \pm ...$$ and where $1 - b_p X + p^4 X^2$ are the Euler factors of the Asai $L$-function of the (unique) Hilbert cusp form of parallel weight $(3, 3)$ and level $1 \cdot \mathcal{O}_F$ attached to the number field $F = \mathbb{Q}(\sqrt{13})$. This strongly suggests that $F$ is a Yoshida lift, and we do not expect to find a corresponding Hermitian modular form. \\

(2) For $-\Delta = 104 = 8 \cdot 13$, the situation is similar: the kernel of the theta map is spanned by an eigenform that appears to be a Yoshida lift of the same Hilbert modular form and the level 8 eigenform discussed in \ref{conj:mainconj}. The missing eigenform for $-\Delta = 143 = 11 \cdot 13$ is also apparently a Yoshida lift. \\

(3) For $-\Delta = 115 = 5 \cdot 23$, one can use Borcherds products to exhibit a Hermitian cusp form of weight four that does not belong to the Maass space. For the theory of Borcherds products, specialized to Hermitian modular forms, we refer to \cite{Dern}. The discriminant group $\mathcal{O}_K' / \mathcal{O}_K \cong \mathbb{Z} / 115\mathbb{Z}$ is cyclic; denoting a generator by $\ell$, one can use the Sage package \cite{WeilRep} to show that there exists a nearly-holomorphic modular form for the appropriate Weil representation whose principal part at $\infty$ is: \begin{align*} &8 e_0 + 5q^{-1/115} (e_{\ell} + e_{-\ell}) + 4q^{-1/115} (e_{24\ell} + e_{-24\ell}) + 5 q^{-4/115} (e_{2\ell} + e_{-2\ell}) - q^{-6/115} (e_{11\ell} + e_{-11 \ell}) \\ &+ 4q^{-9/115} (e_{3\ell} + e_{-3\ell}) + q^{-16/115} (e_{4\ell} + e_{19\ell} + e_{-19\ell} + e_{-4\ell}) +  q^{-24/115} (e_{22\ell} + e_{-22\ell}) \\ &+ q^{-5/23} (e_{5\ell} + e_{-5\ell}) + q^{-29/115} (e_{57\ell} + e_{-57\ell}) + q^{-7/23} (e_{55\ell} + e_{-55\ell}).\end{align*}
The Borcherds product that arises from this is holomorphic of weight four and has vanishing first Fourier--Jacobi coefficient, hence cannot belong to the Maass space. This proves $\mathrm{dim}\, M_4(\Gamma_K) \ge 9$. \\

(4) For $-\Delta_K = 131$ we construct a Borcherds product of weight four by a similar method. If $\ell$ is a generator of the discriminant group $\mathcal{O}_K'/\mathcal{O}_K \cong \mathbb{Z}/131\mathbb{Z}$ then we find a nearly-holomorphic modular form whose principal part at $\infty$ is
\begin{align*} &8e_0 + 6q^{-1/131} (e_{\ell} + e_{-\ell})  + q^{-3/131} (e_{38\ell} + e_{-38\ell}) + 4q^{-4/131} (e_{2\ell} + e_{-2\ell}) + q^{-5/131} (e_{23\ell} + e_{-23\ell}) \\ &+ 3q^{-91/131} (e_{3\ell} + e_{-3\ell}) + q^{-12/131} (e_{55\ell} + e_{-55\ell}) + 2q^{-13/131} (e_{12\ell} + e_{-12\ell}) + 2q^{-15/131} (e_{43\ell} + e_{-43\ell}) \\ &+ q^{-16/131} (e_{4\ell} + e_{-4\ell}) + q^{-25/131} (e_{5\ell} + e_{-5\ell}) + q^{-33/131} (e_{65\ell} + e_{-65\ell}) \\ &+ q^{-35/131} (e_{64\ell} + e_{-64\ell}) + q^{-36/131} (e_{6\ell} + e_{-6\ell}).
\end{align*}
The Borcherds product it produces has weight four and vanishing first Fourier--Jacobi coefficient so it does not belong to the Maass space. \\

The discrepancies for larger $|\Delta_K|$ can apparently be explained similarly. For example, for $-\Delta = 136 = 8 \cdot 17$, there are two algebraic modular forms of weight $0$ that do not correspond to Maass lifts. One has spin character $\spin_{34}$ and is apparently a Yoshida lift (from a Hilbert cusp form of parallel weight $(3, 3)$ for $\mathbb{Q}(\sqrt{17})$ and the CM form of weight 3 and level 8); as for the other form, which has trivial spin character, one can use Borcherds products to show that there is a Hermitian cusp form, not belonging to the Maass space, of weight $4$ for the field $\mathbb{Q}(\sqrt{-136})$.

\end{document}